\documentclass[11pt,a4paper]{article}
\usepackage[latin1]{inputenc}
\usepackage{amsmath}
\usepackage{amsthm}
\usepackage{amsfonts}
\usepackage{amsfonts,amsthm,latexsym,amsmath,amssymb,amscd,epsfig,psfrag,enumerate}
\usepackage{graphics,graphicx, bezier, float, color, hyperref}
\usepackage{amssymb,url}
\usepackage{multienum}
\usepackage[table]{xcolor}
\usepackage{multicol,multirow}
\usepackage{graphicx}
\usepackage{fancyvrb}
\usepackage[none]{hyphenat}
\usepackage{parskip}
\usepackage[toc,page]{appendix}
\usepackage{mathrsfs}
\usepackage[top=2.7cm, bottom=2.7cm, left=1.5cm, right=1.5cm]{geometry}
\usepackage{xcolor}

\makeatletter
\newcommand*{\house}[1]{%
	\mathord{%
		\mathpalette\@house{#1}%
	}%
}
\newcommand*{\@house}[2]{%
	\dimen@=\fontdimen8 %
	\ifx#1\scriptscriptstyle\scriptscriptfont
	\else\ifx#1\scriptstyle\scriptfont
	\else\textfont\fi\fi
	3 %
	\sbox0{%
		$#1%
		\vrule width\dimen@\relax
		\overline{%
			\kern2\dimen@
			\begingroup 
			#2%
			\endgroup
			\kern2\dimen@
		}%
		\vrule width\dimen@\relax
		\mathsurround=1.5\dimen@ 
		$%
	}%
	\ht0=\dimexpr\ht0-\dimen@\relax
	\dp0=\dimexpr\dp0+2\dimen@\relax
	\vbox{%
		\kern\dimen@ 
		\copy0 %
	}%
}

\usepackage{blkarray}
\newtheorem{theorem}{Theorem}[section]

\newtheorem{lemma}{Lemma}[section]

\usepackage[none]{hyphenat}[section]
\newtheorem{definition}{Definition}[section]

\numberwithin{equation}{section}
\numberwithin{table}{section}
\numberwithin{figure}{section}

\title{Product of powers of distinct primes as sums of Fibonacci numbers}
\author{Herbert Batte$^{1,*} $, Florian Luca$^{1}$ and Volker Ziegler$^2$}
\date{}

\begin{document}
\maketitle
\abstract{ Let $F_n$ be the $n$-th Fibonacci number. In this paper, we study the Diophantine equation $F_n+F_m=p^xq^y$ in nonnegative integers $n\ge m$, $x$ and $y$, where $p$ and $q$ are fixed distinct prime numbers. We determine all pairs of primes $(q,p)$ with $q\le \min\{1000,p\}$ such that the above equation has at least two solutions $(x,y)$ (and corresponding $m,n$) in positive 
integers.}

 

\thanks{$ ^{*} $ Corresponding author}

\section{Introduction}
\subsection{Background}

Let $F_n$ denote the $n$-th Fibonacci number defined by $F_0=0$, $F_1=1$ and the recurrence $F_{n+2}=F_{n+1}+F_n$ for all $n\geq 0$. The first few terms are given by 
$$  0,\;1,\;1,\;2,\;3,\;5,\;8,\;13,\;21,\;34,\;55,\;89,\;144,\;\ldots. $$
Closely related to the sequence of Fibonacci numbers is the companion sequence of Lucas numbers $(L_n)_{n\ge 0}$ given by $L_0=2,~L_1=1,~L_{n+2}=L_{n+1}+L_n$ for all $n\ge 0$. Its first terms are 
$$
2,\; 1,\; 3,\; 4,\; 7,\; 11,\; 18,\; 29,\; 47,\; 76,\; 123,\; 199,\; 322,\; \ldots.
$$
There are numerous relations between Fibonacci and Lucas numbers such as 
\begin{equation}
\label{eq:FL1}
L_n^2-5F_n^2=4(-1)^n
\end{equation}
valid for all $n\ge 0$ or
\begin{equation}
\label{eq:FL2}
F_m+F_n=\left\{ \begin{matrix} F_{(m+n)/2}L_{(n-m)/2} & {\text {\rm if}} & m\equiv n\pmod 4;\\
F_{(m-n)/2} L_{(m+n)/2} & {\text{\rm if}} & m\equiv n+2\pmod 4\end{matrix}\right.
\end{equation}
valid for $n\ge m\ge 0$ which are congruent modulo $2$, etc. We shall use such relations freely in our work. 
In \cite{BJJ}, Bravo and Luca treated the Diophantine equation
\begin{equation*}
F_n + F_m = 2^a,
\end{equation*}
and showed that the only solutions $(n,m,a)\in \mathbb{Z}^3$ to this equation with $n>m>0$ are
$$(n,m,a)=(2,1,1),(4,1,2),(4,2,2),(5,4,3),(7,4,4).$$
Using the method due to Bravo and Luca in \cite{BJJ}, Ziegler extended this work without restricting to powers of $2$. That is, in \cite{VZ}, Ziegler looked at the Diophantine equation
\begin{equation}\label{eq:vz}
	F_n+F_m=y^a, \qquad n>m> 1,\;\; a>0,
\end{equation}
with any fixed integer $y$ and investigated the instances when it has at least two solutions in the variable $a$. He showed that 
if $y>1$ is a fixed integer, then there exists at most one solution $a\in \mathbb{Z}$ (and corresponding $m,n$) to the Diophantine equation \eqref{eq:vz} unless $y\in \{ 2,3,4,6,10\}$. In the case that $y\in \{2,3,4,6,10\}$, all solutions $a$ corresponding to $y$ (and some value for the pair $(m,n)$) were given in \cite{VZ} as:
\begin{flushleft}
	\begin{description}
	\item[$y=2$] $(a, n,m)=(2,4,2),(3,5,4),(4,7,4)$;
	\item[$y=3$] $(a,n,m)=(1,3,2),(2,6, 2)$;
	\item[$y=4$] $(a,n,m)=(1,4, 2), (2,7, 4)$;
	\item[$y=6$] $(a, n,m)=(1,5, 2), (2,9, 3)$;
	\item[$y=10$] $(a, n,m)=(1, 6, 3), (2, 16, 7)$.
\end{description}
\end{flushleft}
In this paper, we extend the work in \cite{VZ} and instead study the Diophantine equation 
\begin{equation}\label{eq:main}
	F_n+F_m=p^xq^y, 
\end{equation}
in non-negative integers $x$, $y$ and $n\ge m$, with fixed distinct primes $p$, $q$. We look only at solutions with $xy\ne 0$.  Also, we do not count the number of solutions $(n,m,x,y)$ but rather the number of distinct pairs of nonzero 
integers $(x,y)$ appearing in such an equation. Lastly, since $F_1=F_2=1$, we disregard the case when $n=2$ or $m=2$ when listing actual solutions.

\subsection{Main Result}\label{sec:1.2g}
\begin{theorem}\label{th:main}
Let $p>q$ be fixed distinct primes with $q\le 1000$. Then the Diophantine equation
	\begin{equation*}
		F_n+F_m=p^xq^y, \qquad\text{with}\qquad n\ge m\ge 0, \qquad n,\;m\ne 2\qquad\text{and}\qquad xy\ne 0,
	\end{equation*}
has at least two distinct solutions $(x,y)$ only when $(p,q)\in \mathcal{S}$, where
\begin{align*}
	\mathcal{S}=\{&(3,2),(5,2),(7,2),(7,3),(17,2),(19,2)\}.
\end{align*}	 
In this case, the sums of two Fibonacci numbers that factor as $p^xq^y$ for positive integers $x$ and $y$, with $(p,q)\in \mathcal{S}$, are as follows:
\begin{enumerate}[(i)]
	\item For $(p,q)=(3,2)$, 
	\begin{align*}
		3^x2^y=&F_4 + F_4,\, F_5 + F_1,\,		F_7 + F_5,\, F_8 + F_4,\, F_9 + F_3,\, 
			F_{11} + F_{10} ,\, 
		F_{12} + F_0 ,\\& F_{12} + F_{12},\, F_{13} + F_{10},\, F_{14} + F_{10},\, F_{18} + F_6 .
	\end{align*}
	\item For $(p,q)=(5,2)$, 
	\begin{align*}
		5^x2^y=F_5 + F_5 ,\,	F_6 + F_3 ,\, F_{16} + F_7 ,\,		F_{17} + F_4. 
	\end{align*}
	
	\item For $(p,q)=(7,2)$, 
	\begin{align*}
		7^x2^y&=F_7 + F_1 ,\, F_{10} + F_1 . 
	\end{align*}
	
	\item For $(p,q)=(7,3)$, 
	\begin{align*}
		7^x3^y&=F_7 + F_6 ,\,	F_8 + F_0 ,\, F_{10} + F_6 ,\,		F_{12} + F_4. 
	\end{align*}
	
	\item For $(p,q)=(17,2)$, 
	\begin{align*}
		17^x2^y&=F_8 + F_7 ,\,	F_9 + F_0 ,\, F_{9} + F_9 ,\,		F_{10} + F_7. 
	\end{align*}	

	\item For $(p,q)=(19,2)$, 
	\begin{align*}
		19^x2^y=&F_{10} + F_8,\,	F_{12} + F_6.
	\end{align*}
\end{enumerate}
\end{theorem}

\section{Methods}
\subsection{Preliminaries on Fibonacci and Lucas numbers}
First, we collect some facts on Fibonacci numbers. Let us start with the Binet-formula for the Fibonacci numbers given by
\begin{equation}
\label{eq:F}
F_n=\frac{\alpha^n-\beta^n}{\sqrt{5}}, \qquad\text{where}\qquad
\alpha=\frac{1+\sqrt{5}}{2}\quad\text{and}\quad \beta=\frac{1-\sqrt{5}}2,
\end{equation}
are the roots of the characteristic polynomial $X^2-X-1$ of the Fibonacci sequence. Note that $\beta=-\alpha^{-1}$. The Binet formula for the Lucas numbers is given by 
\begin{equation}
\label{eq:L}
L_n=\alpha^n+\beta^n\qquad {\text{\rm for~all}}\qquad n\ge 0.
\end{equation}
Both the Fibonacci sequence $(F_n)_{n\ge 0}$ and the Lucas sequence $(L_n)_{n\ge 0}$ can be extended to negative integers either by using their recurrence relation or by 
plugging negative numbers $n$ into the Binet formulas \eqref{eq:F} and \eqref{eq:L}. In both cases one checks that $F_{-n}=(-1)^{n-1}F_n$ and $L_{-n}=(-1)^n L_n$ 
hold for all integers $n\ge 0$. Furthermore, all formulas such as \eqref{eq:FL1} and \eqref{eq:FL2} hold when $n,~m$ range through all the integers, positive, zero or negative ones.

Let $n$ be a positive integer. A prime factor $r$ of $F_n$ is called primitive if $r\nmid F_m$ for any positive integer $m<n$. Similarly, a prime factor $r$ of $L_n$ is primitive if 
 $r$ does not divide $L_m$ for any positive integer $m<n$. Since $L_n=F_{2n}/F_n$, it follows that every primitive prime factor of $L_n$ is a prime factor of $F_{2n}$. The following is the celebrated primitive divisor theorem of Carmichael \cite{Car}.  

\begin{theorem}[Carmichael]
\label{thm:Car}
The number $F_n$ has a primitive prime factor for all $n\ge 13$.
\end{theorem}
Listing $F_n$ for $n\in \{1,2,\ldots,12\}$ one notices that $F_n$ has a primitive prime factor for all $n$ except for $n=1,2,6,12$, for which $F_1=F_2=1$, $F_6=8=F_3^3$ and $F_{12}=144=2^43^2=F_3^4F_4^3$, respectively. Further, $L_n$ has a primitive prime factor for all $n=1,6$ for which $L_1=1,~L_6=18=2\cdot 3^2$, while $L_3=2^2,~L_4=3$. 

Next we discuss greatest common divisors of Fibonacci and Lucas numbers. The next result is from the paper \cite{MD} of McDaniel. 

\begin{lemma}[McDaniel, \cite{MD}]
\label{lem:MD}
Let $m,n$ be positive integers. Then:
\begin{itemize}
\item[(i)] $\gcd(F_m,F_n)=F_{\gcd(m,n)}$;
\item[(ii)] $\gcd(L_m,L_n)=L_{\gcd(m,n)}$ if both $m/\gcd(m,n)$ and $n/\gcd(m,n)$ are odd. Otherwise, the above $\gcd$ is $2$ or $1$ according to whether both $m,n$ are multiples of $3$ or not; 
\item[(iii)] $\gcd(L_m,F_n)=L_{\gcd(m,n)}$ if $n/\gcd(m,n)$ is even and $m/\gcd(m,n)$ is odd. Otherwise, the  above $\gcd$ is $2$ or $1$ according to whether both $m,n$ are multiples of $3$ or not. 
\end{itemize}  
\end{lemma}

From the above lemma it follows, in particular, that if $a\mid F_b$ where $a$ and $b$ are both odd, then $a$ is coprime to $L_n$ for any $n\ge 1$. In particular, $L_n$ is never a multiple of $5(=F_5)$ or $13(=F_{7})$ or $17$ (which divides $F_{9}=34$), etc.  This fact will be useful later. 

We will also need the following result about perfect powers in the Fibonacci and Lucas numbers \cite{Buge}.

\begin{theorem}[Bugeaud, Mignotte, Siksek]
\label{thm:BMS}
If $F_n=y^a$ for some integers $y>1$ and $a>1$, then $n\in \{6,12\}$ for which $F_6=2^3,~F_{12}=12^2$. If $L_n=y^a$ for some integers $y>1$ and $a>1$, then $n=3$ for which $L_3=2^2$.
\end{theorem}

In the interest of this paper, we need to discuss the multiplicative properties of numbers of the form 
$$
\delta_{\ell}:=\frac{\alpha^{-\ell}+1}{\sqrt{5}}
$$
for $\ell\ge 2$.  Throughout the paper ${\mathbb K}:={\mathbb Q}(\alpha)$. 
\begin{lemma}
	\label{lem:1}
	We have 
	$$
	N_{{\mathbb K}/{\mathbb Q}}\left(\frac{1+\alpha^{-\ell}}{{\sqrt{5}}}\right)=\left\{\begin{matrix} -L_{\ell}/5 & {\text{\rm if}} & \ell\equiv 1\pmod 2;\\
		F_{\ell/2}^2 & {\text{\rm if}} & \ell\equiv 2\pmod 4;\\
		L_{\ell/2}^2/5 & {\text{\rm if}} & \ell\equiv 0\pmod 4.
	\end{matrix}
	\right.
	$$ 
\end{lemma}

\begin{proof}
	A simple calculation shows that
\begin{align*}
	N_{{\mathbb K}/{\mathbb Q}}\left(\frac{1+\alpha^{-\ell}}{\sqrt{5}}\right)=\frac{(1+\alpha^{-\ell})(1+\beta^{-\ell})}{5}=\frac{L_{-\ell}+1+(-1)^{\ell}}{5}.
\end{align*}
	If
	$\ell$ is odd, the above is $L_{-\ell}/5=-L_{\ell}/5$. If $\ell$ is even, then the above has the value $(L_{\ell}+2)/5$. This is $F_{\ell/2}^2$ or $L_{\ell/2}^2/5$ according to whether $\ell\equiv 2\pmod 4$
	or $\ell\equiv 0\pmod 4$.
\end{proof}

\begin{lemma}
	Assume that $k\ge 2$ and that $\ell_1,\ell_2,\ldots,\ell_k$ are distinct integers $>1$ such that 
	$$
	\frac{1+\alpha^{-\ell_i}}{\sqrt{5}}\qquad i=1,\ldots,k
	$$
	are multiplicatively dependent up to units ($\pm $ powers of $ \alpha$). Then 
	$$
	\{\ell_1,\ldots,\ell_k\}\subseteq \{2,6,10,12\}.
	$$
\end{lemma}

\begin{proof}
	We write a relation of the form 
	$$
	\alpha^a \prod_{i=1}^k \left(\frac{1+\alpha^{-\ell_i}}{{\sqrt{5}}}\right)^{b_i}=\pm 1,
	$$
	for some integers $a,~b_1,\ldots,b_k$, where not all $b_1,\ldots,b_k$ are zero. We take norms and apply Lemma \ref{lem:1}. We get 
	$$
	\prod_{i=1}^k \delta_{\ell_i}^{b_i}=\pm 1,
	$$
	where $\delta_{\ell}=L_{\ell}/5$ for $\ell$ odd, $F_{\ell/2}^2$ for $\ell\equiv 2\pmod 4$ and $L_{\ell/2}^2/5$ for $\ell\equiv 0\pmod 4$. Recall that a primitive prime factor $r$ of $F_n$ is a prime number which does not divide $F_m$ for any positive integer $m<n$. Furthermore, $F_n$ has a primitive prime factor $r$ for all $n>12$ by Theorem \ref{thm:Car}. Since $L_n=F_{2n}/F_n$, it follows that 
	if $n$ is even then the primitive prime factor  of $F_{2n}$ when it exists is a factor of $L_n$. In particular, $L_{\ell}$ contains the primitive prime factors of $F_{2\ell}$ (for $\ell$ odd so $2\ell\equiv 2\pmod 4$) when they exist, $L_{\ell/2}$ contains the primitive prime factors of $F_{\ell}$ (for $\ell\equiv 0\pmod 4$) when they exist, and $F_{\ell/2}$ contains the primitive prime factors of $F_{\ell/2}$ (for $\ell\equiv 2\pmod 4$, so $\ell/2$ is odd), when they exist. In particular, if $\ell_j:=\max\{\ell_i: 1\le i\le k\}>12$, then $\delta_{\ell_j}$ contains a primitive prime factor which does not divide $\delta_{\ell_s}$ for any $s\ne j$. 
	Listing $\delta_{\ell}$ for $\ell=2,\ldots,12$, we get the numbers
	$$
	1,~-\frac{2^2}{5},~\frac{3^2}{5},~-\frac{11}{5},~2^2,~-\frac{29}{5},~\frac{7^2}{5},~-\frac{2^2\cdot 19}{5},~5^2,~-\frac{199}{5},~\frac{2^23^4}{5},\ldots,
	$$
	from where we deduce the desired conclusion. 
\end{proof}

\subsection{Generalities on the equation $F_n+F_m=p^xq^y$}

Assume next that $n\ge1$. Then the Binet formula \eqref{eq:F} immediately yields the inequalities 
\begin{equation}\label{eq:Fib-ieq}
	0.38 \alpha^n <\alpha^n \left(\frac{1-\alpha^{-4}}{\sqrt 5}\right)\leq F_n=\alpha^n\frac{1-(-1)^n \alpha^{-2n}}{\sqrt 5}\leq \alpha^n \left(\frac{1+\alpha^{-4}}{\sqrt 5}\right)< 0.52 \alpha^n.
\end{equation}
Next, assume without loss of generality that $n\ge m\ge 1$, and that $F_n+F_m=p^xq^y$ for some distinct primes $p,q$. Then we have by \eqref{eq:Fib-ieq} that
\begin{equation}\label{eq:2.2g}
	 p^xq^y=F_n+F_m<0.52\alpha^n+
	0.52\alpha^{m}\le 0.52\alpha^n+
	0.52\alpha^{n}<1.1 \alpha^n.
\end{equation}
Taking logarithms on both sides of \eqref{eq:2.2g}, we get
\begin{align*}
	x\log p+y\log q<n\left(\dfrac{\log 1.1}{n}+\log\alpha\right)<0.6n.
\end{align*}
Since $p\ge 2,q\ge 2$, we have
\begin{align}\label{eq:2.4g}
	x<\dfrac{0.6n}{\log p}<n\qquad\text{and}\qquad y<\dfrac{0.6n}{\log q}<n.
\end{align}

Lastly here, we recall one additional simple fact from calculus. If $x\in \mathbb{R}$ satisfies $|x|<1/2$, then 
\begin{align}\label{eq2.5g}
	|\log(1+x)|&<|x-x^2/2+-\dots|
	<|x|+\frac{|x|^2+|x|^3+\dots}2
	<|x|\left(1+\frac{|x|}{2(1-|x|)}\right)<\frac 32 |x|.
\end{align}
In a similar way, we obtain the lower bound $|\log(1+x)|>\frac 12 |x|$ provided that $|x|<1/2$. We use these inequalities frequently throughout the paper.

\subsection{Linear forms in logarithms}
We use several times Baker-type lower bounds for nonzero linear forms in two or more logarithms of algebraic numbers. There are many such bounds mentioned in the literature like that of Baker and W{\"u}stholz from \cite{BW} or Matveev from \cite{MAT}. Before we can formulate such inequalities we need the notion of height of an algebraic number recalled below.  

\begin{definition}\label{def2.1t}
	Let $ \lambda $ be an algebraic number of degree $ d $ with minimal primitive polynomial over the integers $$ a_{0}x^{d}+a_{1}x^{d-1}+\cdots+a_{d}=a_{0}\prod_{i=1}^{d}(x-\lambda^{(i)}), $$ where the leading coefficient $ a_{0} $ is positive. The logarithmic height of $ \lambda$ is given by $$ h(\lambda):= \dfrac{1}{d}\Big(\log a_{0}+\sum_{i=1}^{d}\log \max\{|\lambda^{(i)}|,1\} \Big). $$
\end{definition}
 In particular, if $ \lambda$ is a rational number represented as $\lambda:=p/q$ with coprime integers $p$ and $ q\ge 1$, then $ h(\lambda ) = \log \max\{|p|, q\} $. 
The following properties of the logarithmic height function $ h(\cdot) $ will be used in the rest of the paper without further reference:
\begin{equation}\nonumber
	\begin{aligned}
		h(\lambda_{1}\pm\lambda_{2}) &\leq h(\lambda_{1})+h(\lambda_{2})+\log 2;\\
		h(\lambda_{1}\lambda_{2}^{\pm 1} ) &\leq h(\lambda_{1})+h(\lambda_{2});\\
		h(\lambda^{s}) &= |s|h(\lambda)  \quad {\text{\rm valid for}}\quad s\in \mathbb{Z}.
	\end{aligned}
\end{equation}

A linear form in logarithms is an expression
\begin{equation}
	\label{eq:Lambda}
	\Lambda:=b_1\log \lambda_1+\cdots+b_t\log \lambda_t,
\end{equation}
where for us $\lambda_1,\ldots,\lambda_t$ are positive real  algebraic numbers and $b_1,\ldots,b_t$ are nonzero integers. We assume, $\Lambda\ne 0$. We need lower bounds 
for $|\Lambda|$. We write ${\mathbb L}:={\mathbb Q}(\lambda_1,\ldots,\lambda_t)$ and $D$ for the degree of ${\mathbb L}$.
We start with the general form due to Matveev presented as Theorem 9.4 in \cite{Buge}. 

\begin{theorem}[Matveev, Theorem 9.4 in \cite{Buge}]
	\label{thm:Matg} 
	Put $\Gamma:=\lambda_1^{b_1}\cdots \lambda_t^{b_t}-1=e^{\Lambda}-1$. Assume $\Gamma\ne 0$. Then 
	$$
	\log |\Gamma|>-1.4\cdot 30^{t+3}\cdot t^{4.5} \cdot D^2 (1+\log D)(1+\log B)A_1\cdots A_t,
	$$
	where $B\ge \max\{|b_1|,\ldots,|b_t|\}$ and $A_i\ge \max\{Dh(\lambda_i),|\log \lambda_i|,0.16\}$ for $i=1,\ldots,t$. Furthermore, in this case also
	$$
	\log |\Lambda|>-1.5\cdot 30^{t+3}\cdot t^{4.5} \cdot D^2 (1+\log D)(1+\log B)A_1\cdots A_t,
	$$
	where $\Lambda$ is shown in \eqref{eq:Lambda}.
\end{theorem}
The inequality concerning $\log |\Gamma|$ is from \cite{Buge}. The inequality concerning $\log |\Lambda|$ follows from the one concerning $\log |\Gamma|$ by noting that 
for $|\Lambda|<1/2$, we have that $|\Lambda|>|\Gamma|/2$ (see \eqref{eq2.5g}).

%

\subsection{Reduction methods}

\subsubsection{Continued fractions}

After applying Theorem \ref{thm:Matg}, we get upper bounds on our variables. However,  such upper bounds are too large, thus there is need to reduce them. In this paper, we use the following result related with continued fractions (see Theorem 8.2.4 in \cite{ME}).

\begin{lemma}[Legendre]\label{lem:Legendre} Let $ \mu $ be an irrational number, $[a_0;a_1,a_2,\ldots]$ be the continued fraction expansion of $\mu$. Let $p_i/q_i=[a_0;a_1,a_2,\ldots,a_i]$, for all $i\ge 0$, be all the convergents of the continued fraction of $ \mu$, and $M$ be a positive integer. Let $ N $ be a non--negative integer such that
	$ q_{N} > M $.
	Then putting $ a(M) := \max \{a_{i}: i=0,1,2,\ldots,N   \}$, the inequality
	$$ \bigg| \mu-\frac{r}{s}   \bigg| > \dfrac{1}{(a(M)+2)s^2},  $$
	holds for all pairs $ (r, s) $ of positive integers with $ 0 < s < M $. 
\end{lemma}
However, since there are no methods based on continued fractions to find a lower bound for linear forms in more than two variables with bounded integer coefficients, we use at some point a method based on the LLL--algorithm. We next explain this method in Subsection \ref{sec2.3}.

\subsubsection{Reduced Bases for Lattices and LLL--reduction methods}\label{sec2.3}

Let \( k \) be a positive integer. A subset \( \mathcal{L} \) of the \( k \)--dimensional real vector space \( \mathbb{R}^k \) is called a lattice if there exists a basis \( \{b_1, b_2, \ldots, b_k \}\) of \( \mathbb{R}^k \) such that
\begin{align*}
\mathcal{L} = \sum_{i=1}^{k} \mathbb{Z} b_i = \left\{ \sum_{i=1}^{k} r_i b_i \mid r_i \in \mathbb{Z} \right\}.
\end{align*}
In this situation we say that \( b_1, b_2, \ldots, b_k \) form a basis for \( \mathcal{L} \), or that they span \( \mathcal{L} \). We
call \( k \) the rank of \( \mathcal{L} \). The determinant \( \text{det}(\mathcal{L}) \) of \( \mathcal{L} \) is defined by
\begin{align*}
	 \text{det}(\mathcal{L}) = | \det(b_1, b_2, \ldots, b_k) |,
\end{align*}
with the \( b_i \) being written as column vectors. This is a positive real number that does not depend on the choice of the basis (see \cite{Cas} Sect 1.2).

Given linearly independent vectors \( b_1, b_2, \ldots, b_k \) in \( \mathbb{R}^k \), we refer back to the Gram--Schmidt orthogonalization technique. This method allows us to inductively define vectors \( b^*_i \) (with \( 1 \leq i \leq k \)) and real coefficients \( \mu_{i,j} \) (for \( 1 \leq j \leq i \leq k \)). Specifically,
\begin{align*}
	b^*_i &= b_i - \sum_{j=1}^{i-1} \mu_{i,j} b^*_j,~~~
  \mu_{i,j} = \dfrac{\langle b_i, b^*_j\rangle }{\langle b^*_j, b^*_j\rangle},
\end{align*}
where \( \langle \cdot , \cdot \rangle \)  denotes the ordinary inner product on \( \mathbb{R}^k \). Notice that \( b^*_i \) is the orthogonal projection of \( b_i \) on the orthogonal complement of the span of \( b_1, \ldots, b_{i-1} \), and that \( \mathbb{R}b_i \) is orthogonal to the span of \( b^*_1, \ldots, b^*_{i-1} \) for \( 1 \leq i \leq k \). It follows that \( b^*_1, b^*_2, \ldots, b^*_k \) is an orthogonal basis of \( \mathbb{R}^k \). 
\begin{definition}
The basis \( b_1, b_2, \ldots, b_n \) for the lattice \( \mathcal{L} \) is called reduced if
\begin{align*}
	\| \mu_{i,j} \| &\leq \frac{1}{2}, \quad \text{for} \quad 1 \leq j < i \leq n,~~
	\text{and}\\
	\|b^*_{i}+\mu_{i,i-1} b^*_{i-1}\|^2 &\geq \frac{3}{4}\|b^*_{i-1}\|^2, \quad \text{for} \quad 1 < i \leq n,
\end{align*}
where \( \| \cdot \| \) denotes the ordinary Euclidean length. The constant \( \frac{3}{4} \) above is arbitrarily chosen, and may be replaced by any fixed real number \( \mathcal{P} \) with \( \frac{1}{4} < \mathcal{P} < 1 \),\rm{ (see \cite{LLL} Sect 1)}.
\end{definition}\noindent
Let $\mathcal{L}\subseteq\mathbb{R}^k$ be a $k-$dimensional lattice  with reduced basis $b_1,\ldots,b_k$ and denote by $B$ the matrix with columns $b_1,\ldots,b_k$. 
We define
\[
l\left( \mathcal{L},v\right)= \left\{ \begin{array}{c}
	\min_{u\in \mathcal{L}}||u-v|| \quad  ;~~ v\not\in \mathcal{L}\\
\min_{0\ne u\in \mathcal{L}}||u|| \quad  ;~~ v\in \mathcal{L}
\end{array}
\right.,
\]
where $||\cdot||$ denotes the Euclidean norm on $\mathbb{R}^k$. It is well known that, by applying the
LLL-algorithm, it is possible to give in polynomial time a lower bound for $l\left( \mathcal{L},v\right)\ge c_2$ (see \cite{SMA}, Sect. V.4).
\begin{lemma}\label{lem2.5g}
	Let $v\in\mathbb{R}^k$ and $z=B^{-1}v$ with $z=(z_1,\ldots,z_k)^T$. Furthermore, 
	\begin{enumerate}[(i)]
		\item if $v\not \in \mathcal{L}$, let $i_0$ be the largest index such that $z_{i_0}\ne 0$ and put $\sigma:=\{z_{i_0}\}$, where $\{\cdot\}$ denotes the distance to the nearest integer.
		\item if $v\in \mathcal{L}$, put $\sigma:=1$.
	\end{enumerate}
	Now, define
\[
c_1 := \max_{1\le j\le k}\left\{\dfrac{\|b_1\|}{\|b_j^*\|}\right\}.
\]
Then, $l(\mathcal{L}, y) \ge  c_2 = \lambda\|b_1\|c_1^{-1}$.
\end{lemma}

In our application, we are given real numbers $\eta_0,\eta_1,\ldots,\eta_k$ which are linearly independent over $\mathbb{Q}$ and two positive constants $c_3$ and $c_4$ such that 
\begin{align}\label{2.9}
|\eta_0+a_1\eta_1+\cdots +a_k \eta_k|\le c_3 \exp(-c_4 H),
\end{align}
where the integers $a_i$ are bounded as $|a_i|\le A_i$ with $A_i$ given upper bounds for $1\le i\le k$. We write $A_0:=\max\limits_{1\le i\le k}\{A_i\}$. 

The basic idea in such a situation, from \cite{Weg}, is to approximate the linear form \eqref{2.9} by an approximation lattice. So, we consider the lattice $\mathcal{L}$ generated by the columns of the matrix
$$ \mathcal{A}=\begin{pmatrix}
	1 & 0 &\ldots& 0 & 0 \\
	0 & 1 &\ldots& 0 & 0 \\
	\vdots & \vdots &\vdots& \vdots & \vdots \\
	0 & 0 &\ldots& 1 & 0 \\
	\lfloor M\eta_1\rfloor & \lfloor M\eta_2\rfloor&\ldots & \lfloor M\eta_{k-1}\rfloor& \lfloor M\eta_{k} \rfloor
\end{pmatrix} ,$$
where $M$ is a large constant usually of the size of about $A_0^k$ . Let us assume that we have an LLL--reduced basis $b_1,\ldots, b_k$ of $\mathcal{L}$ and that we have a lower bound $l\left(\mathcal{L},v\right)\ge c_2$ with $v:=(0,0,\ldots,-\lfloor M\eta_0\rfloor)$. Note that $ c_2$ can be computed by using the results of Lemma \ref{lem2.5g}. Then, with these notations the following result  is Lemma VI.1 in \cite{SMA}.
\begin{lemma}[Lemma VI.1 in \cite{SMA}]\label{lem2.6g}
	Let $S:=\displaystyle\sum_{i=1}^{k-1}A_i^2$ and $T:=\dfrac{1+\sum_{i=1}^{k}A_i}{2}$. If $c_2^2\ge T^2+S$, then inequality \eqref{2.9} implies that we either have $a_1=a_2=\cdots=a_{k-1}=0$ and $a_k=-\dfrac{\lfloor M\eta_0 \rfloor}{\lfloor M\eta_k \rfloor}$, or
	\[
	H\le \dfrac{1}{c_4}\left(\log(Mc_3)-\log\left(\sqrt{c_2^2-S}-T\right)\right).
	\]
\end{lemma}

Finally, we present an analytic argument which is Lemma 7 from \cite{GL}. It is useful when obtaining upper bounds on some positive real variable involving powers of the logarithm of the variable itself.
\begin{lemma}[G{\'u}zman \& Luca, \cite{GL}]\label{Lem:Guz} If $ s \geq 1 $, $T > (4s^2)^s$ and $T > \displaystyle \frac{z}{(\log z)^s}$, then $$z < 2^s T (\log T)^s.$$	
\end{lemma}

SageMath 10.6 is used to perform all the computations in this work.

\section{Proof of Theorem \ref{th:main}}\label{sec:proof}

\subsection{Reduction to Zeckendorf representations}
 
We start with a discussion of $n$ versus $m$.

Since we assume that $n\ge m$, we might have $n=m$, in which case
$$
F_n+F_m=2F_n=F_{n+1}+F_{n-2},
$$
or $m=n-1$, in which case we have 
$$
F_n+F_m=F_n+F_{n-1}=F_{n+1}+F_0.
$$
In all other cases $m\le n-2$. We always work with the representations appearing in the right-hand side above. Thus, for us, $d:=n-m\ge 2$. This is called the Zeckendorf representation and is unique up to identifying 
$F_1$ with $F_2$ when $m\in \{1,2\}$. Also, when $m\in \{1,2\}$ we will always work with the value of $m$ such that $n\equiv m\pmod 2$.  

\subsection{An absolute upper bound on $n$ knowing $p$ and $q$}
For the rest of the paper, we note that since $p,q\ge 2$ are distinct primes, then we can assume without loss of generality that $p>q\ge 2$. We first prove the following result.
\begin{lemma}\label{lem3.1g}
	Let $n,m,x,y$ be nonnegative integer solutions to \eqref{eq:main} for fixed primes $p>q\ge 2$ with $n\ge m+2$, then 
	$$n-m<2\cdot 10^{15}(\log p)(\log q) \log n.$$
\end{lemma}
\begin{proof}
We start by rewriting \eqref{eq:main} as
\begin{align*}
	\frac{\alpha^n-\beta^n}{\sqrt{5}}+\frac{\alpha^m-\beta^m}{\sqrt{5}}&=p^xq^y,\\
	\frac{\alpha^n}{\sqrt 5}-p^xq^y&=\frac{-\alpha^m+\beta^n+\beta^m}{\sqrt 5}.
\end{align*}
Dividing through the above relation by $p^xq^y=F_n+F_m>0.38\alpha^n$ and estimating the right-hand side using the fact that $\beta=-\alpha^{-1}$, we get
\begin{align*}
\left|\frac{\alpha^n}{p^xq^y\sqrt{5}}-1\right|&<\frac{\alpha^m+\alpha^{-n}+\alpha^{-m}}{0.38 \alpha^n \sqrt{5}}
=\alpha^{m-n}\left( \frac{1+\alpha^{-(n+m)}+\alpha^{-2m}}{0.38 \sqrt{5}}\right)\\
&\le\alpha^{m-n}\left( \frac{1+\alpha^{-(1+0)}+\alpha^{-2\cdot 0}}{0.38 \sqrt{5}}\right)\\
&< \frac{4 }{\alpha^{d}},
\end{align*}
where $d:=n-m$. So, we conclude that
\begin{align}\label{3.1g}
	\left|p^{-x}q^{-y}\alpha^n(\sqrt{5})^{-1}-1\right| <\frac{4}{\alpha^{d}}.
\end{align}
We now apply Theorem \ref{thm:Matg} to the left-hand side of \eqref{3.1g}. Let $\Gamma_1:=p^{-x}q^{-y}\alpha^n(\sqrt{5})^{-1}-1=e^{\Lambda_1}-1$.
Notice that $\Gamma_1\ne 0$, otherwise we would have $\alpha^n/\sqrt{5}=p^xq^y$. This implies that $\alpha^{2n}\in {\mathbb Q}$, a contradiction  for $n\ge 1$. We use the field ${\mathbb K}$ of
degree $D := 2$. Here, $t := 4$,
\begin{alignat*}{4}
	\lambda_{1} &:= p,         &\quad \lambda_{2} &:= q,         &\quad \lambda_{3} &:= \alpha,     &\quad \lambda_{4} &:= \sqrt{5}, \\
	b_1         &:= -x,        &\quad b_2    &:= -y,         &\quad b_3    &:= n,          &\quad b_4    &:= -1.
\end{alignat*} 
Now, we can write $\max\{|b_1|, |b_2|, |b_3|, |b_4|\} = \max\{x,y,n,1\}=n$ by \eqref{eq:2.4g}, so that we take $B:=n$. Also, 
$$A_i \geq \max\{Dh(\lambda_{i}), |\log\lambda_{i}|, 0.16\}\qquad \text{for all}\qquad i=1,2,3,4.
$$ 
So, 
$$A_1 := Dh(\lambda_{1}) = 2\log p, \quad A_2 := Dh(\lambda_{2}) = 2\log q, \quad A_3 := Dh(\lambda_{3}) = 2\cdot\frac{1}{2} \log\alpha = \log \alpha,
$$
and 
$$A_4 := Dh(\lambda_{4}) = 2\log\sqrt{5} = \log 5.
$$ 
Then, by Theorem \ref{thm:Matg},
\begin{align}\label{3.2g}
	\log |\Gamma_1| &> -1.4\cdot 30^7 \cdot 4^{4.5}\cdot 2^2 (1+\log 2)(1+\log n)(2\log p)(2\log q)(\log \alpha)(\log 5)\nonumber\\
	&> -8.1\cdot 10^{14}(\log p)(\log q) \log n,
\end{align}
where we have used the fact $n\ge 2$. 
Comparing \eqref{3.1g} and \eqref{3.2g}, we get
\begin{align*}
	d:=n-m<2\cdot 10^{15}(\log p)(\log q) \log n.
\end{align*}
This proves Lemma \ref{lem3.1g}.
\end{proof}

Next, we prove the following.
\begin{lemma}\label{lem3.2g}
Let $n,m,x,y$ be nonnegative integer solutions to \eqref{eq:main} for fixed primes $p>q\ge 2$ with $n\ge m+2$, then 
	$$n<10^{35}(\log p)^4(\log q)^2.$$
\end{lemma}
\begin{proof}
	We go back to \eqref{eq:main} and rewrite it as
	\begin{align*}
		\frac{\alpha^n+\alpha^m}{\sqrt 5}-p^xq^y&=\frac{\beta^n+\beta^m}{\sqrt 5}.
	\end{align*}
Like before, we divide through the above relation by $p^xq^y=F_n+F_m>0.38\alpha^n$ and estimate the right-hand side using the fact that $\beta=-\alpha^{-1}$. We get
	\begin{align}\label{lf1}
		\left|\frac{\alpha^n\left(1+\alpha^{-d}\right)}{p^xq^y\sqrt{5}}-1\right|&<\frac{\alpha^{-n}+\alpha^{-m}}{0.38 \alpha^n \sqrt{5}}
		\le\alpha^{-n}\left( \frac{\alpha^{-1}+\alpha^{-0}}{0.38 \sqrt{5}}\right)
		< \frac{2}{ \alpha^{n}}.
	\end{align}
	So, we conclude that
	\begin{align}\label{3.3g}
		\left|p^{-x}q^{-y}\alpha^n\left(\dfrac{1+\alpha^{-d}}{\sqrt5}\right)-1\right| < \frac{2}{\alpha^{n}}.
	\end{align}
	We now apply Theorem \ref{thm:Matg} to the left-hand side of \eqref{3.1g}. Let 
	$$
	\Gamma_2:=p^{-x}q^{-y}\alpha^n\left(1+\alpha^{-d}\right)(\sqrt{5})^{-1}-1=e^{\Lambda_2}-1.
	$$
	Clearly $\Gamma_2\ne 0$, otherwise we would have 
	$$
	\frac{\alpha^n+\alpha^m}{\sqrt{5}}=p^xq^y.
	$$
	Conjugating the above in ${\mathbb K}$ we would get
	$$
	\frac{\alpha^n+\alpha^m}{\sqrt{5}}=-\frac{\beta^n+\beta^m}{\sqrt{5}}.
	$$
	The absolute value of the right-hand side above is $<2/{\sqrt{5}}$, which shows that $\alpha^n<\alpha^n+\alpha^m<2$, a contradiction for $n\ge 2$. 
	 We again use the field ${\mathbb K}$ with $D = 2$, $t = 4$ and 
	\begin{alignat*}{4} 
		\lambda_{1} &:= p,         &\quad \lambda_{2} &:= q,         &\quad \lambda_{3} &:= \alpha,     &\quad \lambda_{4} &:= \dfrac{1+\alpha^{-d}}{\sqrt5}, \\
		b_1         &:= -x,        &\quad b_2    &:= -y,         &\quad b_3    &:= n,          &\quad b_4    &:= 1.
	\end{alignat*}
	Moreover, we still take $B:=n$,
	$A_1 := 2\log p$, $A_2 := 2\log q$ and $A_3 :=\log \alpha$. For the case of $\lambda_{4}$, we compute
	\begin{align*}
	 Dh(\lambda_{4}) = 2h\left(\dfrac{1+\alpha^{-d}}{\sqrt5}\right) &\le 2\log\left(1+\alpha^{-d}\right) + 2\log\sqrt5\\
	&\le 2d\log\alpha+2\log\sqrt{5}+2\log 2\\
	 &< 2\left(2\cdot 10^{15}(\log p)(\log q) \log n\right)\log\alpha+2\log\sqrt{5}+2\log 2\\
	 &<2.1\cdot 10^{15}(\log p)(\log q)\log n.
 \end{align*}
We thus define $A_4:=2.1\cdot 10^{15}(\log p)(\log q) \log n$. So, by Theorem \ref{thm:Matg}, we have 
	\begin{align}\label{3.4g}
		\log |\Gamma_1| &> -1.4\cdot 30^7 \cdot 4^{4.5}\cdot 2^2 (1+\log 2)(1+\log n)(2\log p)(2\log q)(\log \alpha)\cdot 2.1\cdot 10^{15}(\log p)(\log q) \log n\nonumber\\
		&> -6.2\cdot 10^{29}(\log p)^2(\log q)^2 (\log n)^2,
	\end{align}
	where we have used the fact that $n\ge 10$. 
	Comparing \eqref{3.3g} and \eqref{3.4g}, we get
	\begin{align*}
		n<1.3\cdot 10^{30}(\log p)^2(\log q)^2 (\log n)^2.
	\end{align*}
 	We apply Lemma \ref{Lem:Guz} to the above inequality with $ z:=n $, $ s:=2 $ and $T:=1.3\cdot 10^{30}(\log p)^2(\log q)^2$.
Since $T>(4\cdot 2^2)^2$, we get
$$n<2^s T(\log T)^s = 2^2 \cdot 1.3\cdot 10^{30}(\log p)^2(\log q)^2(\log 1.3\cdot 10^{30}(\log p)^4)^2 < 10^{35}(\log p)^4(\log q)^2,$$
which concludes the proof of Lemma \ref{lem3.2g}.
\end{proof}

\subsection{The case for small primes $2\le q<p\le 1000$}
At this stage we have enough to deal with the case $2\le q<p\le 1000$. Then, Lemma \ref{lem3.2g} gives an absolute bound on $n$, that is;
$$n<2\cdot10^{40}, $$
so we need to reduce this bound. Here, we use the LLL-reduction method to find a rather small bound for $n$.

To begin, we go back to equation \eqref{3.1g}. Assuming $d\ge 5$, we can write
\begin{align*}
	|\Lambda_1|=\left|n\log \alpha-x\log p-y\log q-\log \sqrt{5}\right|< \frac{6}{\alpha^{d}},
\end{align*}
where we used \eqref{eq2.5g}. So, for each $p$, $q$ with $2\le q<p\le 1000$, we consider the approximation lattice
$$ \mathcal{A}=\begin{pmatrix}
	1 & 0  & 0 \\
	0 & 1 & 0 \\
	\lfloor M\log (1/p)\rfloor & \lfloor M\log (1/q)\rfloor& \lfloor M\log\alpha \rfloor
\end{pmatrix},$$
with $M:= 8\cdot 10^{120}$ and choose $v:=\left(0,0,-\lfloor M\log (1/\sqrt{5}) \rfloor\right)$. Now, by Lemma \ref{lem2.5g}, we get 
$$c_1=1.9\cdot 10^{-42}\qquad \text{and}\qquad c_2=2.16\cdot 10^{41}.$$
Moreover, by inequalities \eqref{eq:2.4g} and Lemma \ref{lem3.2g}, we have $x,y<n<2\cdot10^{40}$ so
\[
A_i:=2\cdot10^{40},\qquad \text{for}\quad i=1,2,3.
\]
So, Lemma \ref{lem2.6g} gives $S=1.2\cdot 10^{81}$ and $T=3.1\cdot 10^{40}$. Since $c_2^2\ge T^2+S$, then choosing $c_3:=6$ and $c_4:=\log\alpha$, we get $d\le 385$.

Next, we revisit equation \eqref{3.3g} and write
\begin{align*}
	|\Lambda_2|=\left|m\log \alpha-x\log p-y\log q-\log \left(\dfrac{1+\alpha^{-d}}{\sqrt5}\right)\right|< \frac{3}{\alpha^{n}},
\end{align*}
where we used \eqref{eq2.5g} and the assumption that $n\ge 3$. So, for each $p$, $q$ with $2\le q<p\le 1000$, we again consider the approximation lattice
$$ \mathcal{A}=\begin{pmatrix}
	1 & 0  & 0 \\
	0 & 1 & 0 \\
	\lfloor M\log (1/p)\rfloor & \lfloor M\log (1/q)\rfloor& \lfloor M\log\alpha \rfloor
\end{pmatrix},$$
with $M:=  10^{121}$ and $v:=\left(0,0,-\lfloor M\log \left(\sqrt5 /(1+\alpha^{-d})\right) \rfloor\right)$ for each $d\in[0,385]$. So, by Lemma \ref{lem2.5g}, we get 
$$c_1=5.69\cdot 10^{-42}\qquad \text{and}\qquad c_2=4.8\cdot 10^{42}.$$
As before, by inequalities \eqref{eq:2.4g} and Lemma \ref{lem3.2g}, we have $x,y<n<2\cdot10^{40}$ so
\[
A_i:=2\cdot10^{40},\qquad \text{for}\quad i=1,2,3.
\]
Moreover, Lemma \ref{lem2.6g} gives the same values of $S$ and $T$ as before, so choosing $c_3:=3$ and $c_4:=\log\alpha$, we get $n\le 378$. This is for any solution 
$(n,m,x,y)$. 

At this point, we use SageMath to write a simple program to check for all products of prime numbers $p$ and $q$ restricted in the range $2\le q<p\le 1000$ such that the prime factorization $p^xq^y$, with $xy\ne 0$, has at least two representations as $F_n+F_m$, with $0\le m\le  n\le 378$ and $n, m\ne 2$. We find only the representations given in Theorem \ref{th:main}, see Appendix \ref{app1}.

\subsection{The case $q\le 1000$ and $p>1000$}

There are two sub-cases here, that is, the case $q\ne 5$ and $q=5$. We handle them separately. 

\subsubsection{The case $q\ne 5$}\label{subsec:3.4.1}

Here, we assume that we have at least two solutions $(n_i,m_i,x_i,y_i)$ for $i=1,2$. Our purpose is to bound $\min\{n_1,n_2\}$. For this we will first discuss $d_i:=n_i-m_i$ for $i=1,2$. 
If $d_i\ge 5$, we use 
\eqref{eq2.5g} to write \eqref{lf1} as
\begin{equation}
\label{eq:datleast5}
\left|n_i\log\alpha -x_i\log p-y_i\log q-\log \sqrt5\right| < \frac{6}{\alpha^{d_i}}.
\end{equation}
If $d_i=2$, then 
$$
p^{x_i}q^{y_i}=F_{n_i}+F_{n_i-2}=L_{n_i-1}=\alpha^{n_i-1}+\beta^{n_i-1},
$$
which leads to 
$$
|p^{x_i}q^{y_i}\alpha^{-(n_i-1)}-1|=\frac{1}{\alpha^{2(n_i-1)}},
$$
and since $n_i\ge 4$, we get the better inequality 
\begin{equation}
\label{eq:d1is2}
|(n_i-1)\log \alpha-x_i\log p-y_i\log q|<\frac{2}{\alpha^{2(n_i-1)}}.
\end{equation}
The cases $d_i\in \{3,4\}$ yield $q=2,3$, respectively, via the formulas 
$$F_{n_i}+F_{n_i-3}=2F_{n_i-1}\qquad {\text{\rm and}}\qquad F_{n_i}+F_{n_i-4}=3F_{n_i-2}. 
$$
Also $d_i\ne 10$ since 
$$
F_{n_i}+F_{n_i-10}=5L_{n_i-5}
$$
showing that $q=5$, which is not our case. 

Finally, assume that 
$d_1\ge 5,~d_2\ge 5$. We then have
\begin{align*}
	\left|n_1\log\alpha -x_1\log p-y_1\log q-\log \sqrt5\right| &< \frac{6}{\alpha^{d_1}},\\
	\left|n_2\log\alpha -x_2\log p-y_2\log q-\log \sqrt5\right| &< \frac{6}{\alpha^{d_2}},
\end{align*}
Eliminating $\log p$ between the two inequalities above, we get
\begin{equation}\label{eq:lin-f1}
	|\Lambda_3|:=\left|(n_1x_2-n_2x_1)\log \alpha-(x_2y_1-x_1y_2)\log q-(x_2-x_1)\log {\sqrt{5}}\right|<\frac{12n}{\alpha^{\min\{d_1,d_2\}}}.
\end{equation}
In the above elimination, we have used $6(x_2+x_1)<12n$ via \eqref{eq:2.4g}, where we put $n:=\max\{n_1,n_2\}$. 

The linear form \eqref{eq:lin-f1} is nonzero since $q\ne 5$. That is, if it were zero then
\begin{align*}
\alpha^{n_1x_2-n_2x_1}=q^{x_2y_1-x_1y_2} 5^{(x_2-x_1)/2}.	
\end{align*}
Taking norms and absolute values in ${\mathbb K}$, we get $q^{2(x_2y_1-x_1y_2)} 5^{x_2-x_1}=1$, so $x_1=x_2$ and $x_1y_2=x_2y_1$. Thus, also $y_1=y_2$, showing that 
$p^{x_1}q^{y_1}=p^{x_2}q^{y_2}$ so $(n_1,m_1,x_1,y_1)=(n_2,m_2,x_2,y_2)$ by virtue of the uniqueness of the Zeckendorf representation. We use again the field ${\mathbb K}$
of degree $D = 2$, $t = 3$, and the data
	\begin{alignat*}{3}
		\lambda_{1} &:= q,         &\quad \lambda_{2} &:= \alpha,     &\quad \lambda_{3} &:= \sqrt{5}, \\
		b_1         &:= -x_2y_1-x_1y_2, &\quad b_2    &:= (n_1x_2-n_2x_1), &\quad b_3    &:= -(x_2-x_1).
	\end{alignat*}
Next, we write $\max\{|b_1|, |b_2|, |b_3|\} =n^2$ via \eqref{eq:2.4g}, so $B:=n^2$. As before, 
$A_1 := 2\log q$, $A_2 := \log \alpha$ and $A_3 := \log 5$. 
Therefore, by Theorem \ref{thm:Matg},
\begin{align}\label{eq:log-f1}
	\Lambda_3 &> -1.4\cdot 30^6 \cdot 3^{4.5}\cdot 2^2 (1+\log 2)(1+\log n^2)(2\log q)(\log \alpha)(\log 5)\nonumber\\
	&> -6\cdot 10^{12}\log q \log n,
\end{align}
where we have used $n\ge 2$. 
Comparing \eqref{eq:lin-f1} and \eqref{eq:log-f1}, we get
\begin{align*}
	\min\{d_1,d_2\} <  9\cdot 10^{13} \log n,
\end{align*}
for all $q\le 1000$, where $n=\max\{n_1,n_2\}$. This was for $\min\{d_1,d_2\}\ge 5$, but the above inequality holds in the contrary case as well. 

Suppose $d_1\le d_2$. If $d_1\ne 2$, we use \eqref{3.3g} with $(n,m,x,y):=(n_1,m_1,x_1,y_1)$, and \eqref{eq2.5g}, to write 
\begin{equation}
\label{eq:oneoranother}
\left|n_1\log \alpha -x_1\log p-y_1\log q+\log\left( \frac{1+\alpha^{-d_1}}{\sqrt{5}}\right)\right|<\frac{4}{\alpha^{n_1}}.
\end{equation}
If $d_1=2$, we use \eqref{eq:d1is2} namely
\begin{equation}
\label{eq:d1is22}
|(n_1-1)\log \alpha-x_1\log p-y_1\log q|<\frac{2}{\alpha^{2(n_1-1)}}<\frac{4}{\alpha^{n_1}}\qquad (n_1\ge 4).
\end{equation}
We keep 
\begin{equation}
\label{eq:eqford2}
\left|n_2\log\alpha -x_2\log p-y_2\log q-\log \sqrt5\right| < \frac{6}{\alpha^{d_2}}.
\end{equation}
We eliminate $\log p$ between \eqref{eq:oneoranother} (or \eqref{eq:d1is2}) and \eqref{eq:eqford2} (if $d_1\ne 2$,~or $d_1=2$, respectively), getting
\begin{align}
\label{eq:case1}
|\tau_1|&:=\left|(n_1x_2-n_2x_1)\log \alpha-(y_1x_2-y_2x_1)\log q+x_2\log \left(\frac{1+\alpha^{-d_1}}{\sqrt{5}}\right)+x_1\log {\sqrt{5}}\right|\nonumber\\
&<\frac{10}{\alpha^{\min\{d_2,n_1\}}}\qquad (d_1\ge 3),
\end{align}
and
\begin{align}
\label{eq:case2}
|\tau_2|&:=\left|((n_1-1)x_2-n_2x_1)\log \alpha-(y_1x_2-y_2x_1)\log q+x_1\log {\sqrt{5}}\right|<\frac{10}{\alpha^{\min\{d_2,n_1\}}}\qquad (d_1=2),
\end{align}
respectively. Since $q\ne 5$, the  left-hand side of \eqref{eq:case2} is nonzero since otherwise $x_1=0$, which is not allowed. We study whether the left-hand side of \eqref{eq:case1} can be zero.  This leads to
$$
\left(\frac{1+\alpha^{-d_1}}{\sqrt{5}}\right)^{x_2} 5^{x_1/2} q^{-(y_1x_2-y_2x_1)}\alpha^{n_1x_2-x_1n_2}=1.
$$
Taking norms  in ${\mathbb K}$ we get 
$$
\delta_{d_1}^{x_2} 5^{x_1}=\pm q^{2(y_1x_2-x_2y_1)}.
$$
Since $x_1$ and $x_2$ are positive, it follows, by Lemma \ref{lem:1} that $\delta_{d_1}\in \{-L_{d_1}/5,(L_{d_1/2})^2/5\}$  and $x_1=x_2$ (otherwise $5$ is involved at a nonzero exponent in the left-hand side but not in the right-hand side). Next, we get that 
$$
L_{d_1}=q^{2(y_1-y_2)}\qquad {\text{\rm or}}\qquad L_{d_1/2}=q^{y_1-y_2},
$$
according to whether $d_1$ is odd or a multiple of $4$. Since $q\le 1000$, the only possibilities are 
$$(d_1,q)=(3,2),~(4,3),~(8,7),~(16,199),$$ 
and in all cases $y_1-y_2=1$. 
In particular, $y_1\ge 2$. However, in  these cases we have 
$$
F_{n_1}+F_{n_1-d_1}\in \{2F_{n_1-2},3F_{n_1-2},7F_{n_1-4},199F_{n_1-8}\}.
$$
Thus, we must have 
$$
F_{n_1-2}=2^{y_1-1}p^{x_1},~F_{n_1-2}=3^{y_1-1}p^{x_1},~F_{n_1-4}=7^{y_1-1}p^{x_1},~F_{n_1-8}=199^{y_1-1}p^{x_1},
$$
respectively, and in each case $y_1-1>0$.  Thus, $2\mid F_{n_1-2},~3\mid F_{n_1-2},~7\mid F_{n_1-4}$ and $199\mid F_{n_1-8}$, respectively, which shows that  
$$n_1-1=3k_1,\qquad n_1-2=4k_1,\qquad n_1-4=8k_1,\qquad n_1-8=16k_1,$$ 
respectively. Thus,
$$
F_{3k_1},~F_{4k_1},~F_{8k_1},~F_{16k_1},
$$ 
must be $q^{y_1-1} p^{x_1}$, respectively. In the last two cases, we have that both $3,~7$ divide $F_8$, which divides both $F_{8k_1}$ and $F_{16k_1}$, a contradiction.  In the first two cases, $F_{3k_1}$ and 
$F_{4k_1}$ must have a prime factor $p>1000$ but by the primitive divisor theorem they will each have at least two other prime factors 
distinct from $p$ (namely the primitive prime factor of $F_{k_1}$ and of $F_{3k_1}$, or $F_{4k_1}$, respectively).  
This shows that the forms appearing in the left-hand sides of \eqref{eq:case1} and \eqref{eq:case2} are nonzero. 

In \eqref{eq:case1}, we still use the field ${\mathbb K}$
of degree $D = 2$, $t = 4$, and the data
\begin{alignat*}{4} 
	\lambda_{1} &:= \sqrt5,         &\quad \lambda_{2} &:= q,         &\quad \lambda_{3} &:= \alpha,     &\quad \lambda_{4} &:= \dfrac{1+\alpha^{-d_1}}{\sqrt5}, \\
	b_1         &:= x_1,        &\quad b_2    &:= -(y_1x_2-y_2x_1),         &\quad b_3    &:= n_1x_2-n_2x_1,          &\quad b_4    &:= x_2.
\end{alignat*}
Next, $\max\{|b_1|, |b_2|, |b_3|,|b_4|\} =n^2$, so $B:=n^2$, 
$A_1 := \log 5$, $A_2 := 2\log q$ and $A_3 := \log \alpha$. For the case of $\lambda_{4}$, we compute
\begin{align*}
	Dh(\lambda_{4}) = 2h\left(\dfrac{1+\alpha^{-d_1}}{\sqrt5}\right) &\le 2\log\left(1+\alpha^{-d_1}\right) + 2\log\sqrt5\\
	&\le 2d_1 \log\alpha+2\log\sqrt{5}+2\log 2\\
	&< 2\left(8\cdot 10^{13}\log n\right)\log\alpha+2\log\sqrt{5}+2\log 2\\
	&<7.8\cdot 10^{13}\log n.
\end{align*}
We thus define $A_4:=7.8\cdot 10^{13}\log n$.
Therefore, by Theorem \ref{thm:Matg},
\begin{align}\label{eq:log-f4}
	\tau_1 &> -1.4\cdot 30^7 \cdot 4^{4.5}\cdot 2^2 (1+\log 2)(1+\log n^2)(2\log q)(\log \alpha)(\log 5)\cdot 7.8\cdot 10^{13}\log n\nonumber\\
	&> -5\cdot 10^{28}\log q (\log n)^2,
\end{align}
where we have used $n\ge 2$. 
Comparing \eqref{eq:case1} and \eqref{eq:log-f4}, we get
\begin{align*}
	\min\{n_1,d_2\} <  8\cdot 10^{29} (\log n)^2,
\end{align*}
for all $q\le 1000$, with $n=\max\{n_1,n_2\}$. This was for \eqref{eq:case1}, so we re-apply Theorem \ref{thm:Matg} with \eqref{eq:case2}. 

With \eqref{eq:case2}, we use the same field ${\mathbb K}$, $D = 2$, $t = 3$ and
\begin{alignat*}{3} 
	\lambda_{1} &:= \sqrt5,         &\quad \lambda_{2} &:= q,         &\quad \lambda_{3} &:= \alpha,  \\
	b_1         &:= x_1,        &\quad b_2    &:= -(y_1x_2-y_2x_1),         &\quad b_3    &:= (n_1-1)x_2-n_2x_1.
\end{alignat*}
As before, $B:=n^2$, 
$A_1 := \log 5$, $A_2 := 2\log q$ and $A_3 := \log \alpha$. 
Therefore, by Theorem \ref{thm:Matg},
\begin{align}\label{eq:log-f5}
	\tau_2 &> -1.4\cdot 30^6 \cdot 3^{4.5}\cdot 2^2 (1+\log 2)(1+\log n^2)(2\log q)(\log \alpha)(\log 5)\nonumber\\
	&> -6\cdot 10^{12}\log q \log n,
\end{align}
where we have used $n\ge 2$. 
Comparing \eqref{eq:case2} and \eqref{eq:log-f5}, we get
\begin{align*}
	\min\{n_1,d_2\} <  9\cdot 10^{13} \log n,
\end{align*}
for all $q\le 1000$, with $n=\max\{n_1,n_2\}$. Thus, in both cases, we have
\begin{align*}
	\min\{n_1,d_2\} <  8\cdot 10^{29} (\log n)^2.
\end{align*}

If  $n_1=\min\{d_2,n_1\}$, then $n_1<  6\cdot 10^{29} (\log n)^2$. So, combining \eqref{eq:main} and \eqref{eq:Fib-ieq} gives
\begin{align*}
	p^x < p^x q^y = F_n+F_m < 2F_n < 1.1\alpha^n,
\end{align*}
so that for the two solutions $(n_i,m_i,x_i,y_i)$ with $i=1,2$, we have 
\begin{align*}
	x_i\log p < n_i.
\end{align*}
Since $x_i\ge 1$, we have 
\begin{align}\label{eq:p-bound1}
	\log p < n_1<  6\cdot 10^{29} (\log n)^2.
\end{align}

Assume next that $d_2=\min\{n_1,d_2\}$ is bounded. If $d_1=2$ and $d_2=2$, then we get 
$$
L_{n_i-1}=p^{x_i}q^{y_i},\qquad i=1,2,
$$
and this contradicts the primitive divisor theorem. So, if $d_1=2$, then $d_2>2$.

We rewrite \eqref{eq:oneoranother} for $(n_i,m_i,x_i,y_i)$ and $i=1,2$ to get
\begin{gather}
\label{eq:group3} 
	\begin{aligned}
	\left|n_1\log \alpha +\log\left( \frac{1+\alpha^{-d_1}}{\sqrt{5}}\right)-x_1\log p-y_1\log q\right|&<\frac{4}{\alpha^{n_1}},\\
	\left|n_2\log \alpha +\log\left( \frac{1+\alpha^{-d_2}}{\sqrt{5}}\right)-x_2\log p-y_2\log q\right|&<\frac{4}{\alpha^{n_2}}.
\end{aligned}
\end{gather}
Eliminating $\log p$ from the two inequalities in \eqref{eq:group3}, we get
\begin{equation}\label{eq:ineq_a}
|\Lambda_4|:=\left|(n_1x_2-n_2x_1)\log \alpha-(y_1x_2-x_1y_2)\log q+x_2\log\left(\frac{1+\alpha^{-d_1}}{\sqrt{5}}\right)-x_1\log\left(\frac{1+\alpha^{-d_2}}{{\sqrt{5}}}\right)\right|
 < \frac{8n}{\alpha^{\min\{n_1,n_2\}}}.
\end{equation}
Next, we show that the left-hand side of \eqref{eq:ineq_a} is not zero.

Assume $\max\{d_1,d_2\}>1000$. Then the left-hand side of \eqref{eq:ineq_a} is not zero. Indeed, if $d_1\ne d_2$, assume that $d_1>d_2$. Then 
$N_{{\mathbb K}/{\mathbb Q}}((1+\alpha^{-d_1})/{\sqrt{5}})$ contains a primitive divisor (of $F_{2d_1}$ if $d_1$ is odd, of $F_{d_1}$ if $d_1\equiv 0\pmod 4$ and of $F_{d_1/2}$ if 
$d_1\equiv 2\pmod 4$). In any case, this is congruent to $\pm 1$ modulo $d_1$, so it is at least $d_1-1>q$ since $d_1>1000$ and $q<1000$. This shows that the left-hand side is not zero if $d_1\ne d_2$. But the same argument also shows that it is also nonzero 
if $d_1=d_2$ except possibly if $x_1=x_2$ and then $\log((1+\alpha^{-d_1})/{\sqrt{5}})$ does not appear in the left-hand side of \eqref{eq:ineq_a}, but if this is so, then the left-hand side is 
\begin{align*}
	|x_1(n_1-n_2)\log \alpha-x_1(y_1-y_2)\log q|,
\end{align*}
and this is not zero since if it were then also $y_1=y_2$ and $n_1=n_2$, a contradiction. 

It remains to consider the case where $\max\{d_1,d_2\} \le 1000$ (with $d_2>d_1\ge 2$) and $q \le 1000$, with $q \ne 5$. A brute-force computational check over this entire range is highly inefficient. To address this, we employed a refined strategy that checks for multiplicative independence only for specific, promising candidates. This method uses properties of Fibonacci and Lucas numbers to significantly reduce the search space. The linear form on the left-hand side of \eqref{eq:ineq_a} vanishes only when the following multiplicative relation holds for some nontrivial tuple $(\mathfrak{b}_1,\mathfrak{b}_2,x_1,x_2) \in \mathbb{Z}^4 \setminus \{(0,0,0,0)\}$:
\begin{align}\label{eq:multip}
	\alpha^{\mathfrak{b}_1} q^{\mathfrak{b}_2} \left( \frac{1+\alpha^{-d_1}}{\sqrt{5}} \right)^{x_1} \left( \frac{1+\alpha^{-d_2}}{\sqrt{5}} \right)^{-x_2} = 1,
\end{align}
where $\mathfrak{b}_1:=n_1x_2-n_2x_1$ and $\mathfrak{b}_2:=-(y_1x_2-x_1y_2)$. Note that by our earlier assumption that $x_iy_i\ne 0$, we also have that $x_i, y_i\ne0$ while checking the multiplicative independence above.

Our computational check was performed in two parts (see Appendix \ref{app2}):
\begin{enumerate}[(i)]
	\item We first considered all pairs $2 \le d_1 \le 12$ and $d_1<d_2\le 1000$, restricting the prime $q$ to be a prime factor of the associated Fibonacci or Lucas numbers.
	\item We then checked pairs $d_1 > 13$ and $d_2 > d_1$ where $d_1|d_2$ and $d_2/d_1$ is an odd prime. The candidate for the prime $q$ was determined by the presence of a unique ``extra" prime factor in the factorization of the Lucas or Fibonacci number associated with $d_2$ that was not a factor of the number associated with $d_1$.
\end{enumerate}
This optimized computation confirmed the existence of multiplicative dependencies only for the following values:
\begin{align*}
	(d_1,d_2,q)\in \{(2,6,2),(2,14,13),(2,22,89),(2,26,233),(3,12,3), (4,20,41),(6,18,17)\}.
\end{align*}
Recall that any $d_i=2$ leads to $F_n+F_{n-2}=L_{n-1}$, where $L_n$ is the $n$-th Lucas number. Lucas numbers $L_n$'s are never divisible by odd Fibonacci numbers of odd index. That is, $L_n$'s are never divisible by $5(=F_5)$, $13(=F_7)$, $89(=F_{11})$, $233(=F_{13})$. In fact, we have that $\gcd(F_a,F_b)=1$ or $2$ ($2$ when both $a$ and $b$ are multiple of $3$) \textbf{except} when putting $d:=\gcd(a,b)$, we have $a/d$ is even and $b/d$ is odd in which case the above $\gcd$ is exactly $L_d$. In particular, if $a$ is odd, then with $d=\gcd(a,b)$, we have that $a/d$ is odd. So, $\gcd(F_a,L_b)=1,2$, but it cannot be $2$ when $F_a$ is odd. This shows that $\gcd(F_a,L_b)=1$ for all $a$ odd such that $F_a$ is also odd. We now remain with a few candidates as
\begin{align*}
	(d_1,d_2,q)\in &\{(3,12,3), (4,20,41),(6,18,17)\}.
\end{align*}
At this point, we remain with one of them with $d_2=4$ namely $(4,20,41)$. Here we have 
\begin{align*}
	F_n+F_{n-20}&=F_{n-10}L_{10}=3\cdot 41\cdot F_{n-10},\\
	F_m+F_{m-4}&=F_{m-2}L_2=3\cdot F_{m-2},
\end{align*}
so $\{p,q\}=\{3,41\}$ and this case belongs to $p,q$ both $\le 1000$. Lastly, there are still the cases $(3,12,3)$ and $(6,18,17)$. They are similar and give that $q=2$. Indeed $F_n\equiv F_{n-3} \pmod 2$ so if $d_i$ is a multiple of $3$ then $F_n+F_{n-d}$ is even. This gives
$F_n+F_{n-12}=F_{n-6}L_6=2\cdot 3^2 F_{n-6}$, so $\{p,q\}=\{2,3\}$, a case already treated, and lastly
$F_n+F_{n-18}=F_9 L_{n-9}=2\cdot 17\cdot L_{n-9}$, so $\{p,q\}=\{2,17\}$, again a case already treated. Thu,s $\Lambda_4$ never vanishes.

To apply Theorem \ref{thm:Matg} with \eqref{eq:ineq_a}, we again use the field ${\mathbb L}:=\mathbb{Q}(\alpha)$ of
degree $D := 2$. Again, $t := 4$,
\begin{alignat*}{4} 
	\lambda_{1} &:= q,         &\quad \lambda_{2} &:= \alpha,         &\quad \lambda_{3} &:= \frac{1+\alpha^{-d_1}}{{\sqrt{5}}},     &\quad \lambda_{4} &:= \frac{1+\alpha^{-d_2}}{{\sqrt{5}}}, \\
	b_1         &:= -(y_1x_2-x_1y_2),        &\quad b_2    &:= n_1x_2-n_2x_1,         &\quad b_3    &:= x_2,          &\quad b_4    &:= -x_1.
\end{alignat*}
As before, we take $B:=n^2$,
$A_1 := 2\log q$ and $A_2 :=\log \alpha$. For the case of $\lambda_{3}$ and $\lambda_{4}$, we compute
\begin{align*}
	Dh(\lambda_{3}) = 2h\left(\frac{1+\alpha^{-d_1}}{{\sqrt{5}}}\right) &\le 2\log\left(1+\alpha^{-d_1}\right) + 2\log\sqrt5\\
	&\le 2d_1\log\alpha+2\log\sqrt{5}+2\log 2\\
	&< 2\left(8\cdot 10^{13} \log n\right)\log\alpha+2\log\sqrt{5}+2\log 2\\
	&<7.9\cdot 10^{13}\log n,
\end{align*}
and similarly $Dh(\lambda_{4})$.
We thus define $A_3=A_4:=7.9\cdot 10^{13} \log n$. So, by Theorem \ref{thm:Matg}, we have 
\begin{align}\label{eq:log-f2}
	\Lambda_4 &> -1.4\cdot 30^7 \cdot 4^{4.5}\cdot 2^2 (1+\log 2)(1+\log n^2)(2\log q)(\log \alpha)\cdot (7.9\cdot 10^{13} \log n)^2\nonumber\\
	&> -3\cdot 10^{42}\log q (\log n)^3,
\end{align}
where we have used $n\ge 10$. 
Comparing \eqref{eq:ineq_a} and \eqref{eq:log-f2}, we get
\begin{align*}
	\min\{n_1,n_2\} <  2\cdot 10^{43} (\log n)^3,
\end{align*}

Again, combining \eqref{eq:main} and \eqref{eq:Fib-ieq} gives $	p^x < p^x q^y = F_n+F_m < 2F_n < 1.1\alpha^n$,
so for the two solutions $(n_i,m_i,x_i,y_i)$ with $i=1,2$, we have $x_i\log p < n_i$ for $i=1,2$ and hence
\begin{align}\label{eq:p-bound}
	\log p < \min\{n_1,n_2\} <  2\cdot 10^{43} (\log n)^3.
\end{align}
Comparing the bounds in \eqref{eq:p-bound1} and \eqref{eq:p-bound}, we see that in all cases, the bound in \eqref{eq:p-bound} holds, so we use this bound and substitute it in Lemma \ref{lem3.2g} and write 
\begin{align*}
	n <  10^{35}(2\cdot 10^{43} (\log n)^3)^6 < 7\cdot 10^{294}(\log n)^{18}.
\end{align*}
We apply Lemma \ref{Lem:Guz} to the above inequality with $ z:=n $, $ s:=18 $ and $T:=7\cdot 10^{294}$.
Since $T>(4\cdot 18^2)^{18}$, we get
$$n<2^s T(\log T)^s = 2^{18} \cdot 7\cdot 10^{294}(\log 7\cdot 10^{294})^{18} < 10^{352},$$
which gives an absolute bound on $n$.

The bound $n<10^{352}$ is large, so we reduce it. We do this by revisiting \eqref{eq:lin-f1} and applying the LLL-algorithm to obtain a lower bound for the smallest nonzero value of this linear form, constrained by integer coefficients with absolute values not exceeding $n^2 <  10^{704}$. Specifically, we consider the lattice  
\[
\mathcal{A'} = \begin{pmatrix} 
	1 & 0 & 0 \\ 
	0 & 1 & 0 \\ 
	\lfloor M\log \alpha\rfloor & \lfloor M\log (1/q)\rfloor & \lfloor M\log \left(1/\sqrt 5\right) \rfloor
\end{pmatrix},
\]
where we set $M := 3\cdot 10^{2112}$ and $v := (0,0,0)$. Applying Lemma \ref{lem2.5g}, we obtain 
\[
 c_1 =  10^{-709} \quad \text{and} \quad c_2 =1.43\cdot 10^{706}. 
\]
Using Lemma \ref{lem2.6g}, we conclude that $S =2 \cdot 10^{1408}$ and $T = 1.51 \cdot 10^{704}$. Given that $c_2^2 \geq T^2 + S$, and selecting $c_3 := 12n<12\cdot 10^{352}$ and $c_4 := \log \alpha$, we establish the bound $\min\{d_1,d_2\} \leq 8418$. 

As before, assume without loss of generality that $d_1\le d_2$, so that $d_1 \le 8418$. Then, we go to \eqref{eq:case1} and \eqref{eq:case2} (if $d_1\ne 2$,~or $d_1=2$, respectively), we apply the LLL-algorithm to search for a lower bound for the smallest nonzero value of these linear forms. In \eqref{eq:case1}, we have 
\[
\mathcal{A^*} = \begin{pmatrix} 
	1 & 0 & 0 & 0\\ 
	0 & 1 & 0 & 0\\ 
	0 & 0 & 1 & 0\\ 
	\lfloor M\log \alpha\rfloor & \lfloor M\log (1/q)\rfloor & \left\lfloor M\log \left(\frac{1+\alpha^{-d_1}}{\sqrt{5}}\right) \right\rfloor& \left\lfloor M\log \sqrt5 \right\rfloor
\end{pmatrix},
\]
$M := 10^{2817}$ and $v := (0,0,0,0)$. Applying Lemma \ref{lem2.5g} gives $c_2 =4.39\cdot 10^{706}$. 
Using Lemma \ref{lem2.6g}, we get $S =3 \cdot 10^{1408}$ and $T = 2.1 \cdot 10^{704}$. Choosing $c_3 := 10$ and $c_4 := \log \alpha$, we obtain $\min\{n_1,d_2\} \leq 10102$. In \eqref{eq:case2}, we have 
\[
\mathcal{A^{**}} = \begin{pmatrix} 
	1 & 0  & 0\\ 
	0 & 1  & 0\\ 
	0 & 0  & 0\\ 
	\lfloor M\log \alpha\rfloor & \lfloor M\log (1/q)\rfloor & \left\lfloor M\log \sqrt5 \right\rfloor
\end{pmatrix},
\]
$M := 10^{2113}$ and $v := (0,0,0)$. Applying Lemma \ref{lem2.5g} gives $c_2 =8.1\cdot 10^{706}$ and Lemma \ref{lem2.6g} gives $S =2 \cdot 10^{1408}$ and $T = 1.51 \cdot 10^{704}$. Choosing $c_3 := 10$ and $c_4 := \log \alpha$, we obtain $\min\{n_1,d_2\} \leq 6732$. In both cases, $\min\{n_1,d_2\} \leq 10102$.

Next, if $d_2:=\min\{n_1,d_2\}$, then for each $d_1\in [1,8418]$ and $d_2\in [d_1,10102]$, we go to \eqref{eq:ineq_a} and again apply the LLL-algorithm to search for a lower bound for the smallest nonzero value of this linear form, constrained by integer coefficients with absolute values not exceeding $n^2 <  10^{704}$. We consider the lattice  
\[
\mathcal{A''} = \begin{pmatrix} 
	1 & 0 & 0 & 0\\ 
	0 & 1 & 0 & 0\\ 
	0 & 0 & 1 & 0\\ 
	\lfloor M\log \alpha\rfloor & \lfloor M\log (1/q)\rfloor & \left\lfloor M\log \left(\frac{1+\alpha^{-d_1}}{\sqrt{5}}\right) \right\rfloor& \left\lfloor M\log \left(\frac{\sqrt{5}}{1+\alpha^{-d_2}}\right) \right\rfloor
\end{pmatrix},
\]
where we set $M := 4\cdot 10^{2816}$ and $v := (0,0,0,0)$. Applying Lemma \ref{lem2.5g}, we obtain 
\[
 c_1 =  10^{-706} \quad \text{and} \quad c_2 =3.5\cdot 10^{705}. 
\]
Using Lemma \ref{lem2.6g}, we conclude that $S =3 \cdot 10^{1408}$ and $T = 2.1 \cdot 10^{704}$. Given that $c_2^2 \geq T^2 + S$, and selecting $c_3 := 8n<8\cdot 10^{352}$ and $c_4 := \log \alpha$, we establish the bound $\min\{n_1,n_2\} \leq 11790$. 
Thus, similar to how we obtained \eqref{eq:p-bound}, we get that 
\begin{align*}
	\log p \leq 11790,
\end{align*}
and so the bound in Lemma \ref{lem3.2g} gives 
\begin{align}\label{eq:n-b1}
	n \leq 10^{35}(11790)^6< 3\cdot 10^{59}.
\end{align}
Note that if $n_1:=\min\{n_1,d_2\}$, then $n_1\le 10102$ and we even deduce a smaller bound than in \eqref{eq:n-b1}, so we always have \eqref{eq:n-b1}.

We do this LLL-reduction process again using the improved bound on $n$. Again, we revisit \eqref{eq:lin-f1} with the absolute values of integer coefficients not exceeding $n^2 <  10^{119}$. Specifically, we consider the lattice  $\mathcal{A'}$ as before
with $M := 3\cdot 10^{357}$ and $y := (0,0,0)$. Applying Lemma \ref{lem2.5g} gives 
\[
 c_1 =  10^{-119} \quad \text{and} \quad c_2 =6.2\cdot 10^{120}. 
\]
Using Lemma \ref{lem2.6g}, we get that $S =2 \cdot 10^{238}$ and $T = 1.51 \cdot 10^{119}$. Choosing $c_3 := 12n<36\cdot 10^{59}$ and $c_4 := \log \alpha$, we obtain $\min\{d_1,d_2\} \leq 1422$.  

Again, without loss of generality, assume that $d_1\le d_2$, so that $d_1 \le 1422$. Then, we go to \eqref{eq:case1} and consider
$\mathcal{A^*} $ with
$M := 10^{476}$ and $v := (0,0,0,0)$. Lemma \ref{lem2.5g} gives $c_2 =6.7\cdot 10^{120}$ while
 Lemma \ref{lem2.6g} gives $S =2.43 \cdot 10^{238}$ and $T = 1.36 \cdot 10^{119}$. Choosing $c_3 := 10$ and $c_4 := \log \alpha$, we obtain $\min\{n_1,d_2\} \leq 1704$. In \eqref{eq:case2}, we use 
$
\mathcal{A^{**}} $,
$M := 10^{357}$ and $v := (0,0,0)$. Applying Lemma \ref{lem2.5g} gives $c_2 =4.92\cdot 10^{120}$ and Lemma \ref{lem2.6g} gives $S =2.43 \cdot 10^{238}$ and $T = 1.36 \cdot 10^{119}$ as before. Choosing $c_3 := 10$ and $c_4 := \log \alpha$, we obtain $\min\{n_1,d_2\} \leq 1135$. In all cases, $\min\{n_1,d_2\} \leq 1704$.

So, if $d_2:=\min\{n_1,d_2\}$, then for each $d_1\in [1,1422]$ and $d_2\in [d_1,1704]$, we use the linear form in \eqref{eq:ineq_a} and apply the LLL-algorithm to find a lower bound for the smallest nonzero value of this linear form bounded by integer coefficients with absolute values not exceeding $n^2 <  10^{119}$. We use the lattice $\mathcal{A''}$ with $M := 4\cdot 10^{476}$ and $y := (0,0,0,0)$. Applying Lemma \ref{lem2.5g}, we obtain 
\[
 c_1 =  10^{-117} \quad \text{and} \quad c_2 =2.2\cdot 10^{121}. 
\]
By Lemma \ref{lem2.6g}, we conclude that $S =3 \cdot 10^{238}$ and $T = 2.1 \cdot 10^{119}$. Since $c_2^2 \geq T^2 + S$, then choosing $c_3 := 8n<24\cdot 10^{59}$ and $c_4 := \log \alpha$, we get $\min\{n_1,n_2\} \leq 1988$. The case  $n_1:=\min\{n_1,d_2\}$ produces even a much smaller bound as explained in the first reduction before.

To proceed, assume without loss of generality that $n_1:=\min\{n_1,n_2\}<1988$ so that $n_2<3\cdot 10^{59}$ via \eqref{eq:n-b1}. We want to reduce this bound on $n_2$ further. To do this, we perform a computational search for Fibonacci sums that are products of two prime powers, and implement an optimized algorithm to find pairs of primes $(p, q)$ such that $F_{n_1} + F_{m_1} =  p^{x_1}q^{y_1}$ for $n_1 \le 1988$ and $0 \le m_1 \le n_1-2$. A brute-force factorization of every such sum would be computationally prohibitive due to the immense size of the numbers involved.

To circumvent this challenge, we developed a more efficient approach based on primality testing properties. The algorithm first iterates through all valid triples $(n_1, m_1, q)$, where $q$ is a prime less than 1000 but not 5. For each sum $F_{n_1} + F_{m_1}$ that is divisible by $q$, we compute the exponent $y_1 = \nu_q(F_{n_1} + F_{m_1})$ and define the number $B := (F_{n_1} + F_{m_1})/q^{y_1}$. To efficiently check if $B$ is a power of a single prime, we use a quick test based on Fermat's Little Theorem. We compute the greatest common divisor $C := \gcd(\text{mod}(2^B-2, B), B)$ using SageMath's highly optimized \texttt{power\_mod} command. This test quickly filters for candidates where $B$ is likely a prime power, as $C>1$ must hold in such cases. For each candidate, we then perform a definitive check by factoring $C$ and verifying that $B$ is indeed a power of the single prime factor of $C$.

This refined computational process, whose code is provided in Appendix \ref{app3}, identified a total of 5875 instances, confirming that for all found solutions, the exponent $x_1$ of $p$ was exactly 1. So, we now have the pairs $(p,q)$, where $p$ is given by $(F_{n_1}+F_{m_1})/q^{y_1}$ and we can do LLL with these pairs, to bound $n_2$ for any other potential solution.

We now utilize each of these pairs to significantly reduce the upper bound on $n_2$. For a given prime pair $(p,q)$, we consider the linear form in logarithms derived from \eqref{3.3g}, given by
\begin{align*}
	\left|m_2\log \alpha-x_2\log p-y_2\log q-\log \left(\dfrac{1+\alpha^{-d_2}}{\sqrt5}\right)\right|< \frac{3}{\alpha^{n_2}},
\end{align*}
with the absolute values of integer coefficients not exceeding $n_2 <  3\cdot 10^{59}$.
This linear form is then used to construct a three-dimensional approximation lattice $\mathcal{A}$ with a specific basis matrix and a target vector $v$, where we iterate through values of $d_2 \in [2, 1704]$. The LLL algorithm is applied to each of the 5875 pairs independently, a process often referred to as parallel computing. 
 We consider the approximation lattice
$$ \mathcal{A}=\begin{pmatrix}
	1 & 0  & 0 \\
	0 & 1 & 0 \\
	\lfloor M\log (1/p)\rfloor & \lfloor M\log (1/q)\rfloor& \lfloor M\log\alpha \rfloor
\end{pmatrix},$$
where we set $M := 8.1\cdot 10^{178}$ and $v :=\left(0,0,-\lfloor M\log \left(\sqrt5 /(1+\alpha^{-d_2})\right) \rfloor\right)$.
By running the algorithm for all pairs, we find the most restrictive lower bound on the lattice vector length, represented by the minimum value of $c_2 = 8.61 \cdot 10^{61}$ across all pairs and $d_2$ values. Using Lemma \ref{lem2.6g}, we have $S =2.7 \cdot 10^{119}$ and $T = 4.51 \cdot 10^{59}$. Choosing $c_3=3$ and $c_4=\log\alpha$, we obtain a uniform and much-improved upper bound, confirming that for any other potential solution, we must have $n_2 \le 562$.

To conclude this subsection, we use SageMath again to check whether these prime pairs have at least two representations of \eqref{eq:main} with each fixed $p^x q^y$. We do not find any further solutions than those listed in the main result.

\subsubsection{The case $q=5$}\label{subsec:3.4.2}

{\bf The case $m=0$ or $d$ is even}

\medskip

Experience suggests that the solutions with $m=0$ or $d=n-m$ even behave differently than the ones with $d$ odd, so we analyze them first. 
If $m=0$, we get 
\begin{align}\label{eq:Fm=0}
F_n=5^xp^x.
\end{align}
So, $5\mid n$. Write $n=5k$. If $k=5$, we get $F_{n}=F_{25}=5^2\cdot 3001$. So we have the pair $(p,q)=(3001,5)$. So we substitute these $p$, $q$ values in Lemma \ref{lem3.2g} and obtain
\begin{align}\label{eq:m=0}
	n< 1.1\cdot 10^{39},
\end{align}
so we again need to reduce this bound as we did in the case $2\le q<p\le 1000$, but now with $p=3001$ and $q=5$. We still use the LLL-reduction algorithm to reduce this bound on $n$.

Again, we go back to equation \eqref{3.1g} and assume $d\ge 5$, so that
\begin{align*}
	|\Lambda_1|=\left|n\log \alpha-x\log 3001-(y+1/2)\log 5\right|< \frac{6}{\alpha^{d}}.
\end{align*}
So, we consider the approximation lattice
$$ \mathcal{A}=\begin{pmatrix}
	1 & 0  & 0 \\
	0 & 1 & 0 \\
	\lfloor M\log (1/3001)\rfloor & \lfloor M\log (1/5)\rfloor& \lfloor M\log\alpha \rfloor
\end{pmatrix},$$
with $M:= 1.4\cdot 10^{117}$ and choose $v:=\left(0,0,0\right)$. Lemma \ref{lem2.5g} gives
$$c_1=1.7\cdot 10^{-39}\qquad \text{and}\qquad c_2=4.16\cdot 10^{40}.$$
On the other hand, inequalities \eqref{eq:2.4g} and Lemma \ref{lem3.2g} imply $x,\,y+1/2<n<1.11\cdot10^{39}$ so
\[
A_i:=1.11\cdot10^{39},\qquad \text{for}\quad i=1,2,3.
\]
Thus, Lemma \ref{lem2.6g} gives $S=3.64\cdot 10^{78}$ and $T=1.66\cdot 10^{39}$. Since $c_2^2\ge T^2+S$, then choosing $c_3:=6$ and $c_4:=\log\alpha$, we get $d\le 369$.

Next, we go to equation \eqref{3.3g} and write
\begin{align*}
	|\Lambda_2|=\left|m\log \alpha-x\log 3001-y\log 5-\log \left(\dfrac{1+\alpha^{-d}}{\sqrt5}\right)\right|< \frac{3}{\alpha^{n}},
\end{align*}
with the assumption that $n\ge 3$. As before, we consider the approximation lattice
$$ \mathcal{A}=\begin{pmatrix}
	1 & 0  & 0 \\
	0 & 1 & 0 \\
	\lfloor M\log (1/3001)\rfloor & \lfloor M\log (1/5)\rfloor& \lfloor M\log\alpha \rfloor
\end{pmatrix},$$
with $M:=  10^{118}$ and $v:=\left(0,0,-\lfloor M\log \left(\sqrt5 /(1+\alpha^{-d})\right) \rfloor\right)$ for each $d\in[0,369]$. So, by Lemma \ref{lem2.5g}, we get 
$$c_1=7.3\cdot 10^{-41}\qquad \text{and}\qquad c_2=1.43\cdot 10^{40}.$$
In the same way, Lemma \ref{lem2.6g} gives the same values of $S$ and $T$ as before, so choosing $c_3:=3$ and $c_4:=\log\alpha$, we get $n\le 375$, and this is for any solution 
$(n,m,x,y)$. So, as in the case $2\le q<p\le 1000$, we write a simple program in SageMath to check for the prime factorization $3001^x\cdot 5^y$ with $xy\ne 0$, written in two representations as $F_n+F_0=F_n$, with $2\le  n\le 375$. We find no solution to this case (namely, when $m=0$).

In fact, if $5\nmid k$, then $k\ge 13$ since $F_n$ is not divisible by any other prime less than $1000$ except for $5$. But then both 
$F_k$ and $F_{5k}$ have primitive divisors larger than $5$,  so indeed Equation \eqref{eq:Fm=0} is impossible. 

From now on, $m\ge 1$. Assume next that $d$ is even. If $d=2$, then 
$$
F_n+F_{n-2}=L_{n-1}=5^yp^x
$$
has no positive integer solutions with positive $y$ since the Lucas numbers are never multiples of $5$. For $d>2$, we write
$$
F_n+F_{n-d}=F_aL_b,\qquad {\text{\rm where}}\qquad \{a,b\}=\{n-d/2,d/2\}.
$$
In the above representation, $a$ and $b$ are coprime. Indeed, if not, then by Lemma \ref{lem:MD}, we would have by putting $c=\gcd(a,b)$, that $2c\mid a$ and $b/a$ is odd. Thus, 
$$
p^x5^y=F_aL_b=F_c L_c^2 \left(\frac{F_{a}}{F_{2d}}\right)\left(\frac{L_{b}}{L_c}\right).
$$
Further, $c\ge 5$ and $F_c$ and $L_c$ are coprime and $5\nmid L_c$, which shows that the only possibilities are $F_c=5^y,~L_c=p^x$. The primitive divisor theorem now leads to 
 the conclusion that $c=5$, so $p=11$, a case already treated. Thus, $a$ and $b$ are coprime and since $5\nmid L_b$, we must have $F_a=5^y,~L_b=p^x$. By the primitive divisor theorem again and Theorem \ref{thm:BMS}, we get $a=5$ and $x=1$.   Since $n\ge d$ and $p>1000$, we have $d=10$ and $b=n-5$. This shows that $(x,y)=(1,1)$, which makes this the smallest solution 
 $(n,m,x,y)=(n_1,n_1-5,1,1)$. Further, $p>1000$ which gives $n_1>17$. This gives
 $$
 \alpha^{n_1-5}+\beta^{n_1-5}=p,\qquad {\text{\rm so}}\qquad 1-\alpha^{-(n_1-5)}p=-\frac{\beta^{n_1-5}}{\alpha^{n_1-5}}=\pm \frac{1}{\alpha^{2(n_1-5)}},
 $$
 therefore
 \begin{equation}
 \label{eq:101}
 |(n_1-5)\log \alpha-\log p|<\frac{2}{\alpha^{2(n_1-5)}}.
 \end{equation}
 Let $(n_2,m_2,x_2,y_2)$ be a second solution. Clearly, $n_2>n_1$ since not both $x_2,~y_2$ can be $1$. Also, $d_2$ is odd is not equal to $3$ (otherwise $F_{n_2}+F_{n_2-3}$ is even), so $d_2\ge 5$. Thus, similar to \eqref{eq:datleast5}, we get
 \begin{equation}
 \label{eq:102}
 |n_2\log \alpha-x_2 \log p-(y_2+1/2)\log 5|<\frac{6}{\alpha^{d_2}}.
 \end{equation}
 We eliminate $\log p$ between \eqref{eq:101} and \eqref{eq:102} and get
\begin{align}\label{eq:t3}
 |\tau_3|&:=\left|(n_2-x_2(n_1-5))\log \alpha-(2y_2+1)\log \sqrt5\right|\nonumber\\
 &<\frac{2x_2}{\alpha^{2(n_1-5)}}+\frac{6}{\alpha^{d_2}}\le \frac{2n_2+6}{\alpha^{\min\{n_1,d_2\}}}\le \frac{3n_2}{\alpha^{\min\{n_1,d_2\}}}.
\end{align}
 In the above computation, we used that $n_2>n_1>17$, so $2n_2+6<3n_2$ and $2(n_1-5)>n_1$.  The linear form in the left-hand side of \eqref{eq:t3} is non-zero, so we can apply Theorem \ref{thm:Matg}.  
 
We use the field ${\mathbb K}$ of
degree $D := 2$ and for this case, $t := 2$,
\begin{alignat*}{2}
	\lambda_{1} &:= \alpha,         &\quad \lambda_{2} &:= \sqrt5,    \\
	b_1         &:= n_2-x_2(n_1-5),        &\quad b_2    &:= -(2y_2+1).
\end{alignat*} 
Notice that $\max\{|b_1|, |b_2|\} = n_2^2$ by \eqref{eq:2.4g}, so that we take $B:=n_2^2$. On the other hand,  
$$A_1 := Dh(\lambda_{1}) = 2\cdot\frac{1}{2} \log\alpha = \log \alpha,
\qquad\text{and}\qquad A_2 := Dh(\lambda_{2}) = 2\log \sqrt5.
$$ 
So, by Theorem \ref{thm:Matg},
\begin{align}\label{t3}
	\log |\tau_3| &> -1.4\cdot 30^5 \cdot 2^{4.5}\cdot 2^2 (1+\log 2)(1+\log n_2^2)(\log \alpha)(2\log \sqrt5)\nonumber\\
	&> -3\cdot 10^{10}\log n_2,
\end{align}
where we have used the fact $n_2\ge 2$. 
Comparing \eqref{eq:t3} and \eqref{t3}, we get
\begin{align*}
	\min\{n_1,d_2\} <6.5\cdot 10^{10}\log n_2.
\end{align*} 

If $n_1:=\min\{n_1,d_2\}$, we get $n_1< 6.5\cdot 10^{10}\log n_2$ and so
$$\log p<6.5\cdot 10^{10}\log n_2,$$ 
via a similar analysis used to obtain \eqref{eq:p-bound}. Therefore, Lemma \ref{lem3.2g} gives 
\begin{align*}
	n_2 < 10^{35}(\log p)^4(\log q)^2<10^{35}(6.5\cdot 10^{10}\log n_2)^6(\log 5)^2 < 2\cdot 10^{100}(\log n_2)^6.
\end{align*}
Applying Lemma \ref{Lem:Guz} to the above inequality with $ z:=n_2 $, $ s:=6 $ and $T:=2\cdot 10^{100}$, we get
 $$n_2<2^s T(\log T)^s = 2^{6} \cdot 2\cdot 10^{100}(\log 2\cdot 10^{100})^{6} < 10^{117},$$
 which gives an absolute bound on $n_1$ and $n_2$, since $n_1<n_2$.

 If $d_2:=\min\{n_1,d_2\}$, then we pass to the linear form 
\begin{equation}
\label{eq:103}
\left|n_2\log \alpha-x_2\log p-y_2\log 5+\log\left(\frac{1+\alpha^{-d_2}}{\sqrt{5}}\right)\right|<\frac{4}{\alpha^{n_2}}.
\end{equation}
We eliminate $\log p$ between \eqref{eq:101} and \eqref{eq:103} getting
\begin{align}\label{eq:t4}
|\tau_4|:=\left|(n_2-x_2(n_1-5))\log \alpha-y_2\log 5+\log\left(\frac{1+\alpha^{-d_2}}{\sqrt{5}}\right)\right|<\frac{4n_2+6}{\alpha^{n_1}}<\frac{5n_2}{\alpha^{n_1}}.
\end{align}
The left-hand side above is not zero. Indeed, if it were zero, then by taking norms and invoking Lemma \ref{lem:1} together with the fact that $d_2$ is odd, we would get that 
$L_{d_2}$ is a power of $5$, which is not the case. So, $\tau_4\ne 0$.

We still use the field ${\mathbb K}$ with $D = 2$, $t = 3$ and 
\begin{alignat*}{3} 
	\lambda_{1} &:= \alpha,         &\quad \lambda_{2} &:= 5,        &\quad \lambda_{3} &:= \dfrac{1+\alpha^{-d_2}}{\sqrt5}, \\
	b_1         &:= n_2-x_2(n_1-5),        &\quad b_2    &:= -y_2,             &\quad b_3    &:= 1.
\end{alignat*}
Moreover, we still take $B:=n_2^2$,
$A_1 :=\log \alpha$ and $A_2:=2\log 5$. For the case of $\lambda_{3}$, we compute
\begin{align*}
	Dh(\lambda_{3}) = 2h\left(\dfrac{1+\alpha^{-d_2}}{\sqrt5}\right) &\le 2\log\left(1+\alpha^{-d_2}\right) + 2\log\sqrt5\\
	&\le 2d_2\log\alpha+2\log\sqrt{5}+2\log 2\\
	&< 2\left(6.5\cdot 10^{10}\log n_2\right)\log\alpha+2\log\sqrt{5}+2\log 2\\
	&<6.5\cdot 10^{10}\log n_2.
\end{align*}
We thus define $A_3:=6.5\cdot 10^{10}\log n_2$. So, by Theorem \ref{thm:Matg}, we have 
\begin{align}\label{t4}
	\log |\tau_4| &> -1.4\cdot 30^6 \cdot 3^{4.5}\cdot 2^2 (1+\log 2)(1+\log n_2^2)(2\log 5)(\log \alpha)\cdot 6.5\cdot 10^{10}\log n_2\nonumber\\
	&> -3.5\cdot 10^{23} (\log n_2)^2,
\end{align}
where we have used the fact that $n_2> 2$. 
Comparing \eqref{eq:t4} and \eqref{t4}, we get
\begin{align*}
	n_1<8\cdot 10^{23} (\log n_2)^2.
\end{align*}
This means that 
$$\log p<8\cdot 10^{23} (\log n_2)^2,$$ 
via a similar analysis used to obtain \eqref{eq:p-bound}, and Lemma \ref{lem3.2g} gives
\begin{align*}
	n_2 < 10^{35}(\log p)^4(\log q)^2<10^{35}(8\cdot 10^{23} (\log n_2)^2)^4(\log 5)^2 < 7\cdot 10^{178}(\log n_2)^{12}.
\end{align*}
Applying Lemma \ref{Lem:Guz} to the above inequality with $ z:=n_2 $, $ s:=12 $ and $T:=7\cdot 10^{178}$, we get
$$n_2<2^s T(\log T)^s = 2^{12} \cdot 7\cdot 10^{178}(\log 7\cdot 10^{178})^{12} < 10^{214},$$
which also gives an absolute bound on $n_1$ and $n_2$, since $n_1<n_2$. In both cases depending on what $\min\{n_1,d_2\}$ is, we always have 
\begin{align*}
	n_1<n_2<10^{214}.
\end{align*}
We reduce this bound.

We go back to equation \eqref{eq:t3} and write
\begin{align*}
	\left|\dfrac{\log \sqrt5}{\log \alpha}-\dfrac{2(n_2-x_2(n_1-5))}{2y_2+1}\right|
	<  \frac{3n_2}{(2y_2+1)\alpha^{\min\{n_1,d_2\}}\log\alpha}.
\end{align*}
By Lemma \ref{lem:Legendre} with $\mu:=\dfrac{\log \sqrt5}{\log \alpha} $ and $M:=2.1\cdot 10^{214}>2y_2+1>0$, we have 
\begin{align*}
	\dfrac{1}{(a(M)+2)(2y_2+1)^2}<\left|\dfrac{\log \sqrt5}{\log \alpha}-\dfrac{2(n_2-x_2(n_1-5))}{2y_2+1} \right|<
	\frac{3n_2}{(2y_2+1)\alpha^{\min\{n_1,d_2\}}\log\alpha},
\end{align*}
where $a(M)=330$ (in fact, $q_{417}>2.1\cdot 10^{214}$ and $\max\{a_k: 0\le k\le 417\}=330$). The above inequality gives
\begin{align*}
	\dfrac{1}{(330+2)(2y_2+1)^2}&<\dfrac{3\cdot 10^{214}}{(2y_2+1)\alpha^{\min\{n_1,d_2\}}\log\alpha},
\end{align*}
so that 
\begin{align*}
\alpha^{\min\{n_1,d_2\}}&<\dfrac{3\cdot 10^{214}\cdot 332(2y_2+1)}{\log\alpha}<5\cdot 10^{431},
\end{align*}
where we have used the upper bound $2y_2+1< 2.1\cdot 10^{214}$. Taking logarithms of both sides and simplifying, we get $\min\{n_1,d_2\}<2066$.

If $d_2:=\min\{n_1,d_2\}$, then $d_2<2066$, and we go to \eqref{eq:t4} and consider the approximation lattice
$$ \mathcal{A}=\begin{pmatrix}
	1 & 0  & 0 \\
	0 & 1 & 0 \\
	\lfloor M\log \alpha\rfloor & \lfloor M\log (1/5)\rfloor& \left\lfloor M\log\left(\dfrac{1+\alpha^{-d_2}}{\sqrt5}\right)\right\rfloor
\end{pmatrix},$$
with $M:= 10^{1285}$. We choose $v:=\left(0,0,0\right)$. Lemma \ref{lem2.5g} gives $ c_2= 10^{430}$.
On the other hand, inequalities \eqref{eq:2.4g} and Lemma \ref{lem3.2g} imply $y_2,\,n_2-x_2(n_1-5)<n_2^2<10^{428}$, so
\[
A_i:=10^{428},\qquad \text{for}\quad i=1,2,3.
\]
Thus, Lemma \ref{lem2.6g} gives $S=3\cdot10^{856}$ and $T=1.51\cdot 10^{428}$. Since $c_2^2\ge T^2+S$, then choosing $c_3:=5n_2<10^{214}$ and $c_4:=\log\alpha$, we get $n_1\le 5116$. This means that $x_1\log p<n_1<5116$ from \eqref{eq:2.4g}, so that $\log p<5116$ because $x_1>0$.

If $n_1:=\min\{n_1,d_2\}$, then $n_1<2066$ and hence $\log p<2066$. Thus, in both cases depending on what $\min\{n_1,d_2\}$ is, we have $\log p<5116$. Therefore, Lemma \ref{lem3.2g} gives 
\begin{align}\label{n-and-p0}
	n < 10^{35}(\log p)^4(\log q)^2<10^{35}(5116)^4(\log 5)^2 < 2\cdot 10^{50}.
\end{align}
We shall come back to this bound later while concluding this subsection with $q=5$.

\medskip
\noindent{\bf The case $m>0$ and $d$ is odd}

From now on all $d$ are odd and all $m$ are positive. Since in case $m\in \{1,2\}$, we may assume that in fact is taken such that $n\equiv m\pmod 2$ and we have already dealt with such situation (namely, when $d=n-m$ is even), we can assume that $m\ge 3$. Since $d$ is odd and $d\ne 3$ (again otherwise $F_n+F_{n-d}$ is even for $d=3$), it follows that $d\ge 5$ so
\begin{align*}
\left|n\log \alpha-x\log p-(y+1/2)\log 5\right|<\frac{6}{\alpha^d}.
\end{align*}

We take two such inequalities corresponding to $(n_i,x_i,y_i)$ for $i=1,2$. If the matrix
\begin{equation}
	\label{eq:mat}
	\left(\begin{matrix}  n_1 & -x_1 & -(y_1+1/2)\\
		 n_2 & -x_2 & -(y_2+1/2)
		 \end{matrix}\right)
\end{equation}
has rank $2$ (so linearly independent rows), then we eliminate $\log p$ among the two small linear forms obtained from \eqref{3.1g} and get 
\begin{align}\label{eq:lin-f3}
	|\Lambda_5|:=\left|(n_1 x_2-n_2 x_1)\log\alpha - ((2y_1+1)x_2-(2y_2+1)x_1)\log \sqrt5\right| < \dfrac{12n}{\alpha^{\min\{d_1,d_2\}}},
\end{align}
where $n:=\max\{n_1,n_2\}$ and $d_i:=n_i-m_i$, for $i=1,2$. The linear form \eqref{eq:lin-f3} is nonzero since $\alpha$ and $\sqrt5$ are independent. In fact, if it were zero then
\begin{align*}
	\alpha^{n_1x_2-n_2x_1}= 5^{(y_1+1/2)x_2-(y_2+1/2)x_1}.	
\end{align*}
This is possible only if both exponents above are zero. Thus, $5^{(2y_1+1)x_2-(2y_2+1)x_1}=1$, therefore 
$$
\frac{x_1}{x_2}=\frac{2y_1+1}{2y_2+1},\qquad {\text{\rm and also}}\qquad  \frac{n_1}{n_2}=\frac{x_1}{x_2},
$$ 
so the matrix \eqref{eq:mat} has rank $1$, a contradiction. We use as always the field ${\mathbb K}$ of degree $D= 2$ with $t = 2$,
\begin{alignat*}{2}
	\lambda_{1} &:= \sqrt5,         &\quad \lambda_{2} &:= \alpha,     \\
	b_1         &:= - ((2y_1+1)x_2-(2y_2+1)x_1), &\quad b_2    &:= (n_1x_2-n_2x_1).
\end{alignat*}
We define $B:=2n^2$, $A_1 := 2\log \sqrt5$ and $A_2 := \log \alpha$. So, Theorem \ref{thm:Matg} gives
\begin{align}\label{eq:log-f3}
	\Lambda_5 &> -1.4\cdot 30^5 \cdot 2^{4.5}\cdot 2^2 (1+\log 2)(1+\log (2n^2))(2\log \sqrt5)(\log \alpha)\nonumber\\
	&> -3\cdot 10^{10}\log n,
\end{align}
where we have used $n\ge 10$. 
Comparing \eqref{eq:lin-f3} and \eqref{eq:log-f3}, we get
\begin{align*}
	\min\{d_1,d_2\} <  7\cdot 10^{10} \log n,
\end{align*}
which is even a smaller bound than obtained in the case when $q\ne 5$. Assume $\min\{d_1,d_2\}=d_1$. We then take the first equation \eqref{eq:group3}
\begin{equation}
\label{eq:****}
\left|n_1\log \alpha +\log\left( \frac{1+\alpha^{-d_1}}{\sqrt{5}}\right)-x_1\log p-y_1\log  5\right|<\frac{4}{\alpha^{n_1}},
\end{equation}
and 
$$
\left|n_2\log \alpha-x_2\log p-(y_2+1/2)\log 5\right|<\frac{6}{\alpha^{d_2}}.
$$
We multiply the first one by $x_2$ and the second one by $x_1$ and apply the absolute value inequality getting
\begin{align}
\label{eq:22}
|\Lambda_7|&:=\left|(n_1x_2-n_2x_1)\log \alpha-(2y_1x_2-(2y_2+1)x_1)\log \sqrt5+x_2\log\left((1+\alpha^{-d_1})/{\sqrt{5}}\right)\right|\nonumber\\
&<\frac{10n}{\alpha^{\min\{d_2,n_1\}}},
\end{align}
where $n:=\max\{n_1,n_2\}$.
It remains to see that the linear form in the left above is not zero. Note that the coefficient of the number $\log((1+\alpha^{-d_1})/{\sqrt{5}})$ is the positive integer $x_2$, 
so if the above form is zero, we then get that $(1+\alpha^{-d_1})/{\sqrt{5}},~\sqrt5$ and $\alpha$ are multiplicatively independent. 
Taking such a multiplicative relation  and applying norms we get that $N_{{\mathbb K}/{\mathbb Q}}((1+\alpha^{-d_1})/{\sqrt{5}})$ is a power of $5$, which is impossible since 
by Lemma \ref{lem:1}, $d_1\ge 5$ is odd, therefore the above norm is $-L_{d_1}/5$, and $L_{d_1}$ is coprime to $5$. 

To apply Theorem \ref{thm:Matg} to the left-hand side of \eqref{eq:22}, we use the field ${\mathbb K}$
of degree $D = 2$, $t = 3$, and the data
\begin{alignat*}{3}
	\lambda_{1} &:= \sqrt5,         &\quad \lambda_{2} &:= \alpha,     &\quad \lambda_{3} &:= \dfrac{1+\alpha^{-d_1}}{\sqrt{5}}, \\
	b_1         &:= -(2y_1x_2-(2y_2+1)x_1), &\quad b_2    &:= n_1x_2-n_2x_1, &\quad b_3    &:= x_2.
\end{alignat*}
Again, $B:=2n^2$, $A_1 := 2\log \sqrt5$ and $A_2 := \log \alpha$.  For the case of $\lambda_{3}$, we compute
\begin{align*}
	Dh(\lambda_{3}) = 2h\left(\dfrac{1+\alpha^{-d_1}}{\sqrt5}\right) &\le 2\log\left(1+\alpha^{-d_1}\right) + 2\log\sqrt5\\
	&\le 2d_1\log\alpha+2\log\sqrt{5}+2\log 2\\
	&< 2\left(7\cdot 10^{10}\log n\right)\log\alpha+2\log\sqrt{5}+2\log 2\\
	&<6.9\cdot 10^{10}\log n.
\end{align*}
We thus define $A_3:=6.9\cdot 10^{10}\log n$.
Therefore, by Theorem \ref{thm:Matg},
\begin{align}\label{eq:log-f7}
	\Lambda_7 &> -1.4\cdot 30^6 \cdot 3^{4.5}\cdot 2^2 (1+\log 2)(1+\log (2n^2))(2\log \sqrt5)(\log \alpha)\cdot 6.9\cdot 10^{10}\log n \nonumber\\
	&> -4\cdot 10^{23}(\log n)^2,
\end{align}
where we have used $n\ge 10$. 
Comparing \eqref{eq:22} and \eqref{eq:log-f7}, we get
\begin{align*}
	\min\{n_1,d_2\} <  9\cdot 10^{23} (\log n)^2.
\end{align*}

If $n_1:=\min\{n_1,d_2\}$, we get $n_1< 9\cdot 10^{23} (\log n)^2$ and so
$$\log p<9\cdot 10^{23} (\log n)^2,$$ 
via a similar analysis used to obtain \eqref{eq:p-bound}. Therefore, Lemma \ref{lem3.2g} gives 
\begin{align*}
	n < 10^{35}(\log p)^4(\log q)^2<10^{35}(9\cdot 10^{23} (\log n)^2)^6(\log 5)^2 < 2\cdot 10^{179}(\log n)^{12}.
\end{align*}
Applying Lemma \ref{Lem:Guz} to the above inequality with $ z:=n $, $ s:=12 $ and $T:=2\cdot 10^{179}$, we get
$$n<2^s T(\log T)^s = 2^{12} \cdot 2\cdot 10^{179}(\log 2\cdot 10^{179})^{12} < 10^{215}.$$

If $d_2:=\min\{n_1,d_2\}$, then $d_2< 9\cdot 10^{23} (\log n)^2$. We return again to the expression \eqref{eq:ineq_a}, namely
\begin{align}\label{eq:ineq_aa}
	|\Lambda_4|&:=\left|(n_1x_2-n_2x_1)\log \alpha-(y_1x_2-x_1y_2)\log 5+x_2\log\left(\frac{1+\alpha^{-d_1}}{\sqrt{5}}\right)-x_1\left(\frac{1+\alpha^{-d_2}}{{\sqrt{5}}}\right)\right|\nonumber\\
	&< \frac{8n}{\alpha^{\min\{n_1,n_2\}}}.
\end{align}
The left-hand side above is nonzero. Indeed, if $d_1\ne d_2$, this is nonzero by Lemma \ref{lem:1} because $d_1,~d_2$ are both odd and $L_{d_1}$ and $L_{d_2}$ have primitive prime factors which are not $5$. If $d_1=d_2=d$, the above form is again nonzero for the same reason except when $x_2=x_1$ since the coefficient of 
$\log((1+\alpha^{-d})/{\sqrt{5}})$ above is $x_2-x_1$. But in  this case the linear form becomes 
$$
|x_1(n_2-n_1)\log \alpha-x_1(y_1-y_2)\log {\sqrt{5}}|,
$$
and this is not zero except for $y_1=y_2$ and $n_1=n_2$. But this leads to $m_1=m_2$, so $(n_1,m_1,x_1,y_1)=(n_2,m_2,x_2,y_2)$, which is not the case. 
So, we can apply again Matveev's theorem and bound $\min\{n_1,n_2\}$ in \eqref{eq:ineq_aa}. Similar to how we obtained \eqref{eq:log-f2}, the only difference here is $q=5$, $A_3:=6.9\cdot 10^{10}\log n$ and
\begin{align*}
	Dh(\lambda_{4}) = 2h\left(\dfrac{1+\alpha^{-d_2}}{\sqrt5}\right) &\le 2\log\left(1+\alpha^{-d_2}\right) + 2\log\sqrt5\\
	&\le 2d_2\log\alpha+2\log\sqrt{5}+2\log 2\\
	&< 2\left(9\cdot 10^{23}(\log n)^2\right)\log\alpha+2\log\sqrt{5}+2\log 2\\
	&<8.8\cdot 10^{23}(\log n)^2,
\end{align*}
so $A_4:=8.8\cdot 10^{23}(\log n)^2$. Thus, Theorem \ref{thm:Matg} gives
\begin{align}\label{eq:log-f2a}
	\Lambda_4 &> -1.4\cdot 30^7 \cdot 4^{4.5}\cdot 2^2 (1+\log 2)(1+\log n^2)(2\log 5)(\log \alpha)\cdot6.9\cdot 10^{10}\log n\cdot  8.8\cdot 10^{23}(\log n)^2\nonumber\\
	&> -3\cdot 10^{49}(\log n)^4,
\end{align}
where we have used $n\ge 10$. Comparing \eqref{eq:ineq_aa} and \eqref{eq:log-f2a}, we get
\begin{align*}
	\min\{n_1,n_2\} < 6.5\cdot 10^{49}(\log n)^4.
\end{align*}
Depending on whether the minimum above is $n_1$ or $n_2$, we always have from \eqref{eq:2.4g} that
\begin{align*}
	x_1\log p <n_1\qquad\text{and}\qquad x_2\log p< n_2.
\end{align*}
Since $x_1$, $x_2>1$,  we have that $\log p<\min\{n_1,n_2\}<6.5\cdot 10^{49}(\log n)^4$. Therefore, Lemma \ref{lem3.2g} gives 
\begin{align*}
	n < 10^{35}(\log p)^4(\log q)^2<10^{35}(6.5\cdot 10^{49}(\log n)^4)^6(\log 5)^2 < 2\cdot 10^{334}(\log n)^{24}.
\end{align*}
Applying Lemma \ref{Lem:Guz} to the above inequality with $ z:=n $, $ s:=24 $ and $T:=2\cdot 10^{334}$, we get
$$n<2^s T(\log T)^s = 2^{24} \cdot 2\cdot 10^{334}(\log 2\cdot 10^{334})^{24} < 10^{411}.$$

Thus, in all cases from \eqref{eq:22}, we always have 
\begin{align*}
	n<10^{411}.
\end{align*}
We reduce this bound.

We go back to equation \eqref{eq:lin-f3} and write
\begin{align*}
	\left|\dfrac{\log \sqrt5}{\log \alpha}-\dfrac{n_1 x_2-n_2 x_1}{(2y_1+1)x_2-(2y_2+1)x_1}\right|
	<  \frac{12n}{((2y_1+1)x_2-(2y_2+1)x_1)\alpha^{\min\{d_1,d_2\}}\log\alpha}.
\end{align*}
By Lemma \ref{lem:Legendre} with $\mu:=\dfrac{\log \sqrt5}{\log \alpha} $ and $M:=2.1\cdot 10^{822}>2n^2>(2y_1+1)x_2-(2y_2+1)x_1\ne0$ (since we already explained that the matrix \eqref{eq:mat} has rank $2$), we have 
\begin{align*}
	\dfrac{1}{(a(M)+2)((2y_1+1)x_2-(2y_2+1)x_1)^2}&<\left|\dfrac{\log \sqrt5}{\log \alpha}-\dfrac{n_1 x_2-n_2 x_1}{(2y_1+1)x_2-(2y_2+1)x_1}\right|\\
	&<
	\frac{12n}{((2y_1+1)x_2-(2y_2+1)x_1)\alpha^{\min\{d_1,d_2\}}\log\alpha},
\end{align*}
where $a(M)=3435$ (in fact, $q_{1612}>2.1\cdot 10^{822}$ and $\max\{a_k: 0\le k\le 1612\}=3435$). The above inequality gives
\begin{align*}
	\dfrac{1}{(3435+2)((2y_1+1)x_2-(2y_2+1)x_1)^2}&<\dfrac{12\cdot 10^{411}}{((2y_1+1)x_2-(2y_2+1)x_1)\alpha^{\min\{d_1,d_2\}}\log\alpha},
\end{align*}
so that 
\begin{align*}
	\alpha^{\min\{d_1,d_2\}}&<\dfrac{12\cdot 10^{411}\cdot 3437((2y_1+1)x_2-(2y_2+1)x_1)}{\log\alpha}< 10^{1239}.
\end{align*}
Taking logarithms on both sides and simplifying, we get $\min\{d_1,d_2\}<5929$.

Without loss of generality, assume as before that $d_1:=\min\{d_1,d_2\}$. Then $d_1<5929$, so we go to \eqref{eq:22} and consider the approximation lattice
$$\begin{pmatrix}
	1 & 0  & 0 \\
	0 & 1 & 0 \\
	\lfloor M\log \alpha\rfloor & \lfloor M\log (1/5)\rfloor& \left\lfloor M\log\left(\dfrac{1+\alpha^{-d_1}}{\sqrt5}\right)\right\rfloor
\end{pmatrix},$$
with $M:= 10^{2467}$ and choose $v:=\left(0,0,0\right)$. So, Lemma \ref{lem2.5g} gives $ c_2= 10^{825}$.
On the other hand, inequalities \eqref{eq:2.4g} and Lemma \ref{lem3.2g} imply the coefficients of the linear form in \eqref{eq:22} are bounded by $2n^2<2\cdot 10^{822}$ so
\[
A_i:=2\cdot10^{822},\qquad \text{for}\quad i=1,2,3.
\]
Thus, Lemma \ref{lem2.6g} gives $S=1.2\cdot10^{1645}$ and $T=3.1\cdot 10^{822}$. Since $c_2^2\ge T^2+S$, then choosing $c_3:=10n<10^{412}$ and $c_4:=\log\alpha$, we get $\min\{n_1,d_2\}\le 9828$.

If $n_1:=\min\{n_1,d_2\}$, we get $n_1< 9828$ and so
$\log p<9828$. Therefore, Lemma \ref{lem3.2g} gives 
\begin{align*}
	n < 10^{35}(\log p)^4(\log q)^2<10^{35}(9828)^4(\log 5)^2 < 3\cdot 10^{51}.
\end{align*}

If $d_2:=\min\{n_1,d_2\}$, then $d_2< 9828$, and we go to \eqref{eq:ineq_aa} and consider the approximation lattice
\[
 \begin{pmatrix} 
	1 & 0 & 0 & 0\\ 
	0 & 1 & 0 & 0\\ 
	0 & 0 & 1 & 0\\ 
	\lfloor M\log \alpha\rfloor & \lfloor M\log (1/5)\rfloor & \left\lfloor M\log \left(\frac{1+\alpha^{-d_1}}{\sqrt{5}}\right) \right\rfloor& \left\lfloor M\log \left(\frac{\sqrt{5}}{1+\alpha^{-d_2}}\right) \right\rfloor
\end{pmatrix},
\]
with $M:= 10^{3289}$. We choose $v:=\left(0,0,0,0\right)$. Lemma \ref{lem2.5g} gives $ c_2= 10^{827}$.
On the other hand, inequalities \eqref{eq:2.4g} and Lemma \ref{lem3.2g} imply the coefficients of the linear form in \eqref{eq:ineq_aa} are bounded by $n^2<10^{822}$, therefore
\[
A_i:=10^{822},\qquad \text{for}\quad i=1,2,3,4.
\]
Thus, Lemma \ref{lem2.6g} gives $S=4\cdot10^{1644}$ and $T=2.1\cdot 10^{822}$. Since $c_2^2\ge T^2+S$, then choosing $c_3:=8n<10^{412}$ and $c_4:=\log\alpha$, we get $\min\{n_1,n_2\}\le 13752$. So, depending on whether this minimum is $n_1$ or $n_2$, we shall always have $\log p\le 13752$, as explained before. Thus, Lemma \ref{lem3.2g} gives 
\begin{align*}
	n < 10^{35}(\log p)^4(\log q)^2<10^{35}(13752)^4(\log 5)^2 < 9.3\cdot 10^{51}.
\end{align*}

Hence, in all cases,
\begin{align}\label{n-and-p1}
\log p\le 13752\qquad\text{and}\qquad	n < 9.3\cdot 10^{51}.
\end{align}

The trouble is when the rows in \eqref{eq:mat} are linearly dependent. This is something that has not appeared before so let us treat it. Write 
\begin{align*}
(n_i,-x_i,-(y_i+1/2))=\lambda_i(a,b,c),\qquad i=1,2,
\end{align*}
where $a,b,c$ are some fixed positive integers with $\gcd(a,b,c)=1$. Note that $c$ is odd. So, $\lambda_i$ are half integers of numerators $\le 2n_i$ for $i=1,2$. 
We now rewrite \eqref{eq:main} for $(n_i,x_i,y_i)=(n,x,y)$ as 
\begin{align*}
\frac{\alpha^n}{\sqrt{5}}-p^x5^y=\frac{\beta^n}{{\sqrt{5}}}-\frac{\alpha^m-\beta^m}{\sqrt{5}}.
\end{align*}
Dividing through by $\alpha^n/\sqrt5$, we get
\begin{align*}
1-\alpha^{-n}p^x 5^{y+1/2}=\frac{\beta^n}{\alpha^n}-\frac{\alpha^m-\beta^m}{\alpha^n}.
\end{align*}
Let 
\begin{align*}
z:=1-\alpha^{-n}p^x 5^{y+1/2}=-\frac{1}{\alpha^{n-m}}+\frac{\beta^m}{\alpha^n}+\frac{1}{\alpha^{2n}}=-\frac{1}{\alpha^{d}}+\frac{(-1)^m}{\alpha^{d+2m}}+\frac{\beta^n}{\alpha^n}=-\frac{1}{\alpha^{d}}+\mathcal{O}_2\left(\frac{1}{\alpha^{n+m}}\right).
\end{align*}
The constant inside the $\mathcal{O}_2$ symbol above can be taken to be $2$. Note that since $d\ge 5$, $n\ge 17$ (because $F_n+F_m=p^x5^y$ with $y\ge 1$, $x\ge 1$ and $p>1000$), we have that $|z|<0.1$, so 
$1/(1-|z|)<10/9$. Note also that $|z|<3/\alpha^d$. Putting 
\begin{align*}
\Lambda_6 :=n\log \alpha-x\log p-(y+1/2)\log 5,
\end{align*}
we get that 
\begin{align*}
\Lambda_6 & =-\log(1-z)=z+\mathcal{O}_{1}\left(\frac{z^2}{2}\left(1+|z|+\cdots\right)\right)=z+\mathcal{O}_1\left(\frac{z^2}{2(1-|z|)}\right)\\
& =-\frac{1}{\alpha^d}+\frac{(-1)^m}{\alpha^{2m+d}}+\mathcal{O}_1\left(\frac{1}{\alpha^{m+n}}\right)+\mathcal{O}_5\left(\frac{1}{\alpha^{2d}}\right)=-\frac{1}{\alpha^d}+\frac{(-1)^m}{\alpha^{2m+d}}+\mathcal{O}_6\left(\frac{1}{\alpha^{\min\{2d,m+n\}}}\right)\\
& =-\frac{1}{\alpha^d}+\mathcal{O}_7\left(\frac{1}{\alpha^{\min\{2m+d,2d,m+n\}}}\right).
\end{align*}
The numbers $1$, $5$, $6$ and $7$ above represent the size of the implied constants. We now give values $z=z_i$ when $(n,x,y)=(n_i,x_i,y_i)$ for $i=1,2$  and divide by $\lambda_i$, which are 
$\ge 1/2$ getting
\begin{align*}
a\log \alpha-b\log p-c\log 5=\frac{1}{\lambda_i }\left(-\frac{1}{\alpha^{d_i}}+\frac{(-1)^{m_i}}{\alpha^{2m_i+d_i}}\right)+\mathcal{O}_{12} \left(\frac{1}{\alpha^{\min\{2d_i,n_i+m_i\}}}\right),\qquad i=1,2.
\end{align*}
We thus get 
\begin{align*}
\left| \frac{1}{\lambda_1\alpha^{d_1}}-\frac{1}{\lambda_2\alpha^{d_2}}\right|<\frac{28}{\alpha^{\min\{2d_1,2d_2,n_1+m_1,n_2+m_2\}}}.
\end{align*}
Assuming $d_1\le d_2$, we get 
\begin{equation}
	\label{eq:lbd}
	\left|\frac{1}{\lambda_1}-\frac{1}{\lambda_2\alpha^{d_2-d_1}}\right|<\frac{28}{\alpha^{\min\{d_1,2d_2-d_1,2m_1,2m_2+d_2-d_1\}}}.
\end{equation}
Note that the left-hand side above is not zero. Indeed, it can only be zero first if $d_1=d_2$ (since $\alpha^{d_2-d_1}$ must equal the rational number $\lambda_1/\lambda_2$), and later if $\lambda_1=\lambda_2$ but then $(n_1,x_1,y_1+1/2)=(n_2,x_2,y_2+1/2)$, which is false. 
We now incorporate $1/\lambda_2\alpha^{d_2-d_1}$ in the right-hand side getting
\begin{align*}
\frac{1}{n_1}\le \frac{1}{\lambda_1}\le \frac{30}{\alpha^{\min\{d_2-d_1,d_1,2m_1,2m_2+d_2-d_1\}}}.
\end{align*}
So 
\begin{align*}
\min\{d_2-d_1,d_1,2m_1,2m_2+d_2-d_1\}&\le \frac{\log(30n)}{\log \alpha}
<7.1+2.1\log n\\
&\le \left(2.1+\frac{7.1}{\log n}\right) \log n<5\log n.
\end{align*}
Here, $n=\max\{n_1,n_2\}$ and we used the fact that $n\ge 17$. 
So, either 
\begin{align*}
	d_1\le 5\log n,\qquad {\text{\rm or}}\qquad d_2-d_1\le 5\log n,\qquad {\text{\rm or}}\qquad \min\{2m_1,2m_2+d_2-d_1\}\le 5\log n.
\end{align*}
The first condition is what we wanted. If it holds, we can proceed as we did before by moving to inequality \eqref{eq:****}. 
Let us assume that the second condition holds. We return to \eqref{eq:lbd} and multiply it by 
\begin{align*}
\left|\frac{1}{\lambda_1}-\frac{1}{\lambda_2\beta^{d_2-d_1}}\right|\le 2+2\alpha^{d_2-d_1}\le 4\alpha^{d_2-d_1},
\end{align*}
getting 
\begin{align*}
\left|\frac{1}{\lambda_1}-\frac{1}{\lambda_2\alpha^{d_2-d_1}}\right|\left|\frac{1}{\lambda_1}-\frac{1}{\lambda_2\beta^{d_2-d_1}}\right| \le \frac{112\alpha^{d_2-d_1}}{\alpha^{\min\{d_1,d_2,2m_1,2m_2+d_2-d_1\}}}.
\end{align*}
The left-hand side is a rational number which is nonzero and at least $1/({\text{\rm num}}(\lambda_1){\text{\rm num}}(\lambda_2))^2\ge 1/(16n^4)$. We thus get that 
\begin{align*}
\min\{d_1,d_2,2m_1,2m_2+d_2-d_1\} & \le \frac{\log(16\cdot 112 n^4)}{\log \alpha}+d_2-d_1<15.6+8.4\log n+5\log n\\
& <\left(13.4+\frac{15.6}{\log n}\right)\log n<19\log n. 
\end{align*}
If this is $\min\{d_1,d_2\}$, then again we proceed to \eqref{eq:****}. So, assume that
\begin{align*}
\min\{2m_1,2m_2+d_2-d_1\}\le 19\log n.
\end{align*}
But this was under the assumption that $d_2-d_1\le 5\log n$. If this doesn't hold, then the last condition holds namely
\begin{align*}
\min\{2m_1,2m_2+d_2-d_1\}\le 5\log n.
\end{align*}
If the minimum is in $2m_2+d_2-d_1$, then both $d_2-d_1<5\log n$ and $2m_2<5\log n$ hold. Finally, let us assume that 
\begin{align*}
2m_1<5\log n.
\end{align*}
We then incorporate this into \eqref{eq:lbd} using the more precise formula for $z:=z_1$ getting
\begin{align*}
\left|\frac{1}{\lambda_1}\left(-1+\frac{(-1)^{m_1}}{\alpha^{2m_1}}\right)\right|\le \frac{26}{\alpha^{\min\{d_2-d_1,2m_1,2m_2+d_2-d_1\}}}.
\end{align*}
The left-hand side is nonzero since $m_1\ne 0$. Indeed, it could be zero only if $\alpha^{2m_1}\in {\mathbb Q}$, giving $m_1=0$, which is not allowed. We multiply both sides above by the nonzero number 
\begin{align*}
\left|1+\frac{(-1)^{m_1}}{\beta^{2m_1}}\right|\le 1+\alpha^{2m_1}\le 2\alpha^{2m_1},
\end{align*}
getting 
\begin{align*}
\frac{1}{\lambda_1} \left|\left(1+\frac{(-1)^{m_1}}{\alpha^{2m_1}}\right)\left(1+\frac{(-1)^{m_1}}{\beta^{2m_1}}\right)\right|\le \frac{52\alpha^{2m_1}}{\alpha^{d_2-d_1}}.
\end{align*}
Since $1/\lambda_1\ge 1/2n_1$ and the factor multiplying it in the left-hand side above is an integer, we get that
\begin{align*}
d_2-d_1\le \frac{\log(52\alpha^{2m_1})}{\log \alpha}\le 9+2m_1\le 9+6\log n<10\log n.
\end{align*}
So, in all cases 
\begin{align*}
d_2-d_1<10\log n,\qquad \min\{2m_1,2m_2\}\le 19\log n.
\end{align*}
We take this one step further. Assume $2m_1<2m_2+d_2-d_1$. Incorporating $(-1)^{m_1}/\alpha^{2m_1}$ into the left-hand side of our approximations we get
\begin{equation}
\label{eq:110}
\left|\frac{1}{\lambda_1}\left(\frac{1}{\alpha^{d_1}}-\frac{(-1)^{m_1}}{\alpha^{d_1+2m_1}}\right)-\frac{1}{\lambda_2\alpha^{d_2}}\right|\le \frac{26}{\alpha^{\min\{2d_1,2d_2,2m_2+d_2-d_1\}}}.
\end{equation}
This gives 
\begin{align*}
\left|\frac{1}{\lambda_1}\left(1-\frac{(-1)^{m_1}}{\alpha^{2m_1}}\right)-\frac{1}{\lambda_2\alpha^{d_2-d_1}}\right|\le \frac{26}{\alpha^{\min\{d_1,d_2,2m_2+d_2-d_1\}}}.
\end{align*}
We show that the left-hand side is nonzero. Indeed, assume it is zero. Then 
\begin{align*}
(1-(-1)^{m_1}\alpha^{-2m_1})=\pm \frac{\lambda_1}{\lambda_2} \alpha^{d_1-d_2}.
\end{align*}
We take norms and use the fact that $d_1-d_2$ is even (since both $d_1$ and $d_2$ are odd), to get that
\begin{align*}
(1-(-1)^{m_1}\alpha^{-2m_1})(1-(-1)^{m_1}\beta^{2m_1})=\left(\frac{\lambda_1}{\lambda_2}\right)^2 (\alpha \beta)^{d_1-d_2} =\left(\frac{\lambda_1}{\lambda_2}\right)^2.
\end{align*}
The left-hand side is $2-(-1)^{m_1}L_{-2m_1}=2-(-1)^{m_1}L_{2m_1}$. Since $m_1\ge 3$, $L_{2m_1}>2$ so since the right-hand side above is positive, we must have that $m_1$ is odd. Thus,
\begin{align*}
\left(\frac{\lambda_1}{\lambda_2}\right)^2=L_{2m_1}+2=5F_{m_1}^2,
\end{align*}
a contradiction. This shows that the left-hand side in \eqref{eq:110} is nonzero. We now multiply the inequality \eqref{eq:110} with 
\begin{align*}
\left|\frac{1}{\lambda_1}\left(1-\frac{(-1)^{m_1}}{\beta^{2m_1}}\right)-\frac{1}{\lambda_2\beta^{d_2-d_1}}\right|.
\end{align*}
The size of this last number is at most
\begin{align*}
\le 2(\alpha^{2m_1}+1)+2\alpha^{d_2-d_1}\le 5\alpha^{\max\{2m_1,d_2-d_1\}}.
\end{align*}
We get 
\begin{align*}
\left|\frac{1}{\lambda_1}\left(\frac{1}{\alpha^{d_1}}-\frac{(-1)^{m_1}}{\alpha^{d_1+2m_1}}\right)-\frac{1}{\lambda_2\alpha^{d_2}}\right|\left|\frac{1}{\lambda_1}\left(1-\frac{(-1)^{m_1}}{\beta^{2m_1}}\right)-\frac{1}{\lambda_2\beta^{d_2-d_1}}\right|\le \frac{130\alpha^{\max\{2m_1,d_2-d_1\}}}{\alpha^{\min\{d_1,d_2,2m_2+d_2-d_1\}}}.
\end{align*}
The left-hand side above is nonzero and at least $\ge 1/(16n^4)$. We thus get
\begin{align*}
\min\{d_1,d_2,2m_2+d_2-d_1\}&\le \frac{\log(16\cdot 130 n^4)}{\log \alpha}+\max\{2m_1,d_2-d_1\}\\
&\le 16+8.4\log n+19\log n<34\log n.
\end{align*}
A similar inequality is obtained for $\min\{d_1,d_2,2m_1\}$ assuming that $2m_2+(d_2-d_1)<19\log n$ (we just incorporate $1/\alpha^{2m_2+d_2}$ into the right-hand side 
of our estimates). 

If 
\begin{align*}
\min\{d_1,d_2\}<34\log n,
\end{align*}
then we return to \eqref{eq:22} and bound $\min\{d_2,n_1\}$ and later $n_1$ and also $n_2$. Doing this still showed that the bound in \eqref{n-and-p1} still held. So, the only case left is the following.
Let us summarize our findings in Lemma \ref{lem:lem3.3}.  
\begin{lemma}
\label{lem:lem3.3}
	Assume that $q=5,~p>1000$, $(n_1,x_1,y_1),~(n_2,x_2,y_2)$ are  two solutions such that matrix \eqref{eq:mat} has rank $1$, $d_1$ and $d_2$ are odd, $m_1\ge 3,~m_2\ge 3$ and $\min\{d_1,d_2\}\ge 34\log n$. We then have
	\begin{align*}
	d_2-d_1\le 10\log n,\qquad \min\{m_1,m_2\}\le 10\log n,\qquad \max\{m_1,m_2\}<17\log n.
\end{align*}
\end{lemma}

We are almost there. We now look again at the equation
\begin{align*}
F_n+F_m=p^x q^y .
\end{align*}
The left-hand side is 
\begin{align*}
\frac{\alpha^n-(-1)^n \alpha^{-n}}{\sqrt{5}}+F_m=\frac{\alpha^{-n}}{\sqrt{5}} \left(\alpha^{2n}+{\sqrt{5}}F_m\alpha^n-(-1)^{m+1}\right),
\end{align*}
where we used the fact that $d\equiv 1\pmod 2$ so $n\equiv m+1\pmod 2$. We factor the trinomial
\begin{align*}
t^2-{\sqrt{5}}F_m t-(-1)^{m+1}=(t-\zeta_{1,m})(t-\zeta_{2,m}),
\end{align*}
where
\begin{align*}
\zeta_{1,m}=\frac{{\sqrt{5}}F_m+{\sqrt{5F_m^2+4(-1)^{m+1}}}}{2},\qquad \zeta_{2,m}=(-1)^m \zeta_{1,m}^{-1}.
\end{align*}
Note that $5F_m^2+4(-1)^{m+1}=L_m^2+8(-1)^{m+1}$ is not a perfect square since $m\ge 3$, so $L_m\ge 4$ (if this was a square of the same parity as $L_m$ it should be either at least $(L_m+2)^2=L_m^2+4L_m+4$, and $4L_m+4>8$, or at most $(L_m-2)^2=L_m^2-(4L_m-4)$, and $4L_m-4>8$ for $m\ge 3$). So, we may write 
$$
5F_m^2+8(-1)^{m+1}=\lambda_m \square,
$$
where $\lambda_m>1$ is squarefree and is coprime to $5$. We do this for $m_1,m_2$ and work in 
$${\mathbb L}:={\mathbb Q}({\sqrt{5}},{\sqrt{\lambda_{m_1}}},{\sqrt{\lambda_{m_2}}}),$$
which is a field of degree $4$ or $8$.   Inside this field, there is a prime ideal $\pi$ sitting above $p\mid F_{n_i}+F_{m_i}$ for $i=1,2$ such that
$$
\alpha^{n_i}\equiv \pm \beta_{1,m_i}^{\delta_i}\pmod \pi,\qquad \delta_i\in \{\pm 1\},\quad i=1,2.
$$
Since $n_i=d_i+m_i$, we get that 
\begin{equation}
\label{eq:n1minusn2}
|n_1-n_2|=|d_1-d_2+m_1-m_2|\le d_2-d_1+|m_2-m_1|\le 27\log n.
\end{equation}
Further,
$$
\alpha^{n_1-n_2}\equiv \pm \beta_{1,m_1}^{\delta_1} \beta_{1,m_2}^{-\delta_2}\pmod \pi.
$$ 
It follows that $\pi$ divides the algebraic integer
$$
w:=\alpha^{n_1-n_2}\pm \beta_{1,m_1}^{\delta_1}\beta_{1,m_2}^{-\delta_2}.
$$
This is an algebraic  integer since $\alpha,~\beta_{1,m_1},~\beta_{1,m_2}$ are all units. We will justify later that this is nonzero. Assuming that we have proved that, let us see how we can use this information to bound $p$. We have
$$
p\mid N_{{\mathbb L}/{\mathbb Q}} (\pi)\le N_{{\mathbb L}/{\mathbb Q}}(w).
$$
Now $w$ has $4$ or $8$ conjugates each at most as large as
\begin{equation}
\label{eq:maxw}
2\max\{\alpha^{|n_1-n_2|},\beta_{1,m_1}\beta_{1,m_2}\}.
\end{equation}
Note that
$$
\beta_{1,m_i}<{\sqrt{5F_{m_i}^2+8}}<\left(5\alpha^{2(m_i-1)}\left(1+\frac{8}{5\alpha^{2(m_i-1)}}\right)\right)^{1/2}<\alpha^{m_i+1/2},\qquad i=1,2,
$$
where we used the fact that $5<\alpha^4$, $F_{m_i}<\alpha^{m_i-1}$ and $8/(5\alpha^{2(m_i-1)})<1/\alpha$ since $m_i\ge 3$. Thus,
\begin{equation}
\label{eq:maxb1b2}
\beta_{1,m_1}\beta_{1,m_2}<\alpha^{m_1+m_2+1}<\alpha^{27\log n+1}.
\end{equation}
Thus, the house of $w$ is bounded above, by \eqref{eq:n1minusn2}, \eqref{eq:maxw} and \eqref{eq:maxb1b2} by 
$$
2\alpha^{27\log n+1}. 
$$
Thus,
$$
p\le (2\alpha^{27\log n+1})^8,
$$
giving
$$
\log p\le (8\cdot 27 \log n)\log \alpha+8\log 2<104\log n+6.
$$
Now, we can go back to \eqref{lf1} and use Theorem \ref{thm:Matg} to the left-hand side of \eqref{3.3g} with $q=5$ and obtain as in Lemma \ref{lem3.2g} that
\begin{align*}
	n<10^{35}(\log p)^4(\log 5)^2<10^{35}(104\log n+6)^4(\log 5)^2<4.2\cdot 10^{43}(\log n)^4.
\end{align*} 
	We apply Lemma \ref{Lem:Guz} to the above inequality with $ z:=n $, $ s:=4 $ and $T:=4.2\cdot 10^{43}$.
We get
$$n<2^s T(\log T)^s = 2^4 \cdot 4.2\cdot 10^{43}(\log 4.2\cdot 10^{43})^4 < 10^{53},$$
and
$\log p<104\log n+6<104\log 10^{53}+6$.
Therefore
\begin{align}\label{n-and-p2}
	\log p\le 12697\qquad\text{and}\qquad	n <  10^{53}.
\end{align}

It remains to justify that $w$ is nonzero. Assuming that it is we get the relation
\begin{equation}
	\label{eq:last}
	\alpha^{n_1-n_2}=\pm \beta_{1,m_1}^{\delta_1}\beta_{1,m_2}^{-\delta_2},\qquad \delta_1,\delta_2\in \{\pm 1\}.
\end{equation}
Since $\alpha,~\beta_{1,m_1},~\beta_{2,m_2}$ are positive, the sign must be $+$. We distinguish two cases according to $\lambda_{m_1}$ and $\lambda_{m_2}$. For example, it could be that $\lambda_{m_1}=\lambda_{m_2}(=\lambda)$. This happens exactly when  ${\mathbb L}$ has degree $4$. We apply the Galois automorphism of ${\mathbb L}$ 
which maps ${\sqrt{5}}$ to $-{\sqrt{5}}$ and ${\sqrt{\lambda}}$ to $-{\sqrt{\lambda}}$. Under this automorphism, each of $\beta_{1,m_1},~\beta_{1,m_2}$ goes to its negative, so the right--hand side of equation \eqref{eq:last} doesn't change but $\alpha^{n_1-n_2}$ goes to $\beta^{n_1-n_2}$. So, we must have $\alpha^{n_1-n_2}=\beta^{n_1-n_2}$ leading to $\alpha^{2(n_1-n_2)}=\pm 1$, which shows that $n_1=n_2$. A similar contradiction is reached if $\lambda_{m_1}$  and $\lambda_{m_2}$ are different. Since they are different, squarefree, larger than $1$ and coprime to $5$, ${\mathbb L}$ 
has degree $8$ and the map sending each of ${\sqrt{5}},~{\sqrt{\lambda_{m_1}}}$ and ${\sqrt{\lambda_{m_2}}}$ to their negatives is a Galois automorphism of ${\mathbb L}$. Applying this map  to our equation \eqref{eq:last}, we get again that the right--hand side doesn't change but the left--hand side maps to $\beta^{n_1-n_2}$. We get the same conclusion namely that $n_2=n_1$. But then 
$$
\beta_{1,m_1}^{\delta_1}=\beta_{1,m_2}^{\delta_2}
$$
and since $\beta_{1,m_i}>1$ for $i=1,2$, it follows that $\delta_1=\delta_2$. Looking at the coefficient of ${\sqrt{5}}$ in both sides of the equation  we get $F_{m_1}=F_{m_2}$ so 
$m_1=m_2$. Since also $n_1=n_2$, we get  $(n_1,m_1,x_1,y_1)=(n_2,m_2,x_2,y_2)$, a contradiction.

Now, comparing \eqref{n-and-p0}, \eqref{n-and-p1} and \eqref{n-and-p2}, we always have that 
\begin{align}\label{eq:n-bound}
	\log p\le 13752\qquad\text{and}\qquad	n < 10^{53},
\end{align}
regardless on whether the rows of \eqref{eq:mat} are linearly dependent or not, and regardless whether $m=0$, $m>0$, $d$ is odd or even.

We need to find a way of handling these huge bounds for effective computation, such that we complete this case when $q=5$. We go back to Lemma \ref{lem:lem3.3} and proceed in two cases, that is, when $\min\{d_1,d_2\}<34\log n$ and when $\min\{d_1,d_2\}\ge 34\log n$.

\medskip

\noindent{\bf The case $\min\{d_1,d_2\}<34\log n$}

\medskip

Suppose without loss of generality (as before) that $d_1\le d_2$. Since $n<10^{53}$, then $d_1:=\min\{d_1,d_2\}<34\log (10^{53})< 4150$. Recall that the case $m=0$ is not possible here (we already explained this at the start of Subsection \ref{subsec:3.4.2} under the case $m=0$ or $d$ is even. In fact, we reduced the bound in \eqref{eq:m=0} and computationally found no further solutions in this case). Therefore, $m\ge 1$. 

If $d$ is even, we are in the case where we deduced equation \eqref{eq:t3}, so we write
\begin{align*}
	\left|\dfrac{\log \sqrt5}{\log \alpha}-\dfrac{2(n_2-x_2(n_1-5))}{2y_2+1}\right|
	<  \frac{3n_2}{(2y_2+1)\alpha^{\min\{n_1,d_2\}}\log\alpha}.
\end{align*}
Using Lemma \ref{lem:Legendre} with $\mu:=\dfrac{\log \sqrt5}{\log \alpha} $ and $M:=2.1\cdot 10^{53}>2y_2+1>0$, we have 
\begin{align*}
	\dfrac{1}{(a(M)+2)(2y_2+1)^2}<\left|\dfrac{\log \sqrt5}{\log \alpha}-\dfrac{2(n_2-x_2(n_1-5))}{2y_2+1} \right|<
	\frac{3n_2}{(2y_2+1)\alpha^{\min\{n_1,d_2\}}\log\alpha},
\end{align*}
where $a(M)=29$ (that is, $q_{110}>2.1\cdot 10^{53}$ and $\max\{a_k: 0\le k\le 110\}=29$). The above inequality gives
\begin{align*}
	\dfrac{1}{(29+2)(2y_2+1)^2}&<\dfrac{3\cdot 10^{53}}{(2y_2+1)\alpha^{\min\{n_1,d_2\}}\log\alpha},
\end{align*}
so that 
\begin{align*}
	\alpha^{\min\{n_1,d_2\}}&<\dfrac{3\cdot 10^{53}\cdot 31(2y_2+1)}{\log\alpha}<4.1\cdot 10^{108},
\end{align*}
where we have used the upper bound $2y_2+1< 2.1\cdot 10^{53}$. Taking logarithms of both sides, we get $\min\{n_1,d_2\}<520$.

If $d_2:=\min\{n_1,d_2\}$, then $d_2<520$ so we go to \eqref{eq:t4} and consider the approximation lattice
$$ \begin{pmatrix}
	1 & 0  & 0 \\
	0 & 1 & 0 \\
	\lfloor M\log \alpha\rfloor & \lfloor M\log (1/5)\rfloor& \left\lfloor M\log\left(\dfrac{1+\alpha^{-d_2}}{\sqrt5}\right)\right\rfloor
\end{pmatrix},$$
with $M:= 3\cdot 10^{106}$ and choose $v:=\left(0,0,0\right)$. So, Lemma \ref{lem2.5g} gives $ c_2= 6.93\cdot 10^{110}$.
On the other hand, inequalities \eqref{eq:2.4g} and Lemma \ref{lem3.2g} imply $y_2,\,n_2-x_2(n_1-5)<n_2^2<10^{106}$, so
\[
A_i:=10^{106},\qquad \text{for}\quad i=1,2,3.
\]
Thus, Lemma \ref{lem2.6g} gives $S=3\cdot10^{212}$ and $T=1.51\cdot 10^{106}$. Since $c_2^2\ge T^2+S$, then choosing $c_3:=5n_2<5\cdot 10^{53}$ and $c_4:=\log\alpha$, we get $n_1\le 236$. This means that $x_1\log p<n_1<236$ from \eqref{eq:2.4g}, so that $\log p<236$ because $x_1>0$.

If $n_1:=\min\{n_1,d_2\}$, then $n_1<520$ and hence $\log p<520$. Thus, in both cases depending on what $\min\{n_1,d_2\}$ is, we have $\log p<520$. Therefore, Lemma \ref{lem3.2g} gives 
\begin{align}\label{case:nb1}
	n < 10^{35}(\log p)^4(\log q)^2<10^{35}(520)^4(\log 5)^2 < 2\cdot 10^{46}.
\end{align}
 We shall come back to this bound later.
 
If $m>0$ and $d$ is odd, then we go to \eqref{eq:22} and consider the approximation lattice
$$\begin{pmatrix}
	1 & 0  & 0 \\
	0 & 1 & 0 \\
	\lfloor M\log \alpha\rfloor & \lfloor M\log (1/5)\rfloor& \left\lfloor M\log\left(\dfrac{1+\alpha^{-d_1}}{\sqrt5}\right)\right\rfloor
\end{pmatrix},$$
with $M:= 10^{319}$, $v:=\left(0,0,0\right)$ and $d_1\in [5,4150]$, odd. So, Lemma \ref{lem2.5g} gives $ c_2= 10^{108}$.
On the other hand, inequalities \eqref{eq:2.4g} and Lemma \ref{lem3.2g} imply that the coefficients of the linear form in \eqref{eq:22} are bounded by $2n^2<2\cdot 10^{106}$ so
\[
A_i:=2\cdot 10^{106},\qquad \text{for}\quad i=1,2,3.
\]
Now, Lemma \ref{lem2.6g} gives $S=1.2\cdot10^{213}$ and $T=3.1\cdot 10^{106}$. Since $c_2^2\ge T^2+S$, then choosing $c_3:=10n<10^{54}$ and $c_4:=\log\alpha$, we get $\min\{n_1,d_2\}\le 1268$.

If $n_1:=\min\{n_1,d_2\}$, we get $n_1< 1268$ and so
$\log p<1268$. Therefore, Lemma \ref{lem3.2g} gives 
\begin{align*}
	n < 10^{35}(\log p)^4(\log q)^2<10^{35}(1268)^4(\log 5)^2 < 7\cdot 10^{47}.
\end{align*}

If $d_2:=\min\{n_1,d_2\}$, then $d_2< 1268$ so we go to \eqref{eq:ineq_aa} and consider the approximation lattice
\[
\begin{pmatrix} 
	1 & 0 & 0 & 0\\ 
	0 & 1 & 0 & 0\\ 
	0 & 0 & 1 & 0\\ 
	\lfloor M\log \alpha\rfloor & \lfloor M\log (1/5)\rfloor & \left\lfloor M\log \left(\frac{1+\alpha^{-d_1}}{\sqrt{5}}\right) \right\rfloor& \left\lfloor M\log \left(\frac{\sqrt{5}}{1+\alpha^{-d_2}}\right) \right\rfloor
\end{pmatrix},
\]
with $M:= 10^{425}$, $v:=\left(0,0,0,0\right)$ and $d_1\in [5,4150]$, $d_2\in [d_1,1268]$, both odd. So, Lemma \ref{lem2.5g} gives $ c_2= 10^{110}$.
On the other hand, inequalities \eqref{eq:2.4g} and Lemma \ref{lem3.2g} imply the coefficients of the linear form in \eqref{eq:ineq_aa} are bounded by $n^2<10^{106}$ so
\[
A_i:=10^{106},\qquad \text{for}\quad i=1,2,3,4.
\]
So, Lemma \ref{lem2.6g} gives $S=4\cdot10^{212}$ and $T=2.1\cdot 10^{106}$. Since $c_2^2\ge T^2+S$, then choosing $c_3:=8n<8\cdot 10^{53}$ and $c_4:=\log\alpha$, we get $\min\{n_1,n_2\}\le 1765$. So, depending on whether this minimum is $n_1$ or $n_2$, we shall always have $\log p\le 1765$, as explained before. Therefore, Lemma \ref{lem3.2g} gives 
\begin{align*}
	n < 10^{35}(\log p)^4(\log q)^2<10^{35}(1765)^4(\log 5)^2 < 5\cdot 10^{48}.
\end{align*}

Thus, in all cases,
\begin{align}\label{nb-fin}
	\log p\le 1765\qquad\text{and}\qquad	n < 5\cdot 10^{48}.
\end{align}
So, in this case when $\min\{d_1,d_2\}< 34\log n$, we conclude from \eqref{case:nb1} and \eqref{nb-fin} that always $n_1\le 1765$ and $n_2 < 5\cdot 10^{48}$.

To further refine the upper bound on $n_2$, we adopted a computational approach that mirrors our conclusion in Subsection \ref{subsec:3.4.1}. Again, we write an efficient search algorithm to identify  primes $p$ satisfying the equation $F_{n_1} + F_{m_1} = p^{x_1}5^{y_1}$ for the specified ranges $n_1 \le 1765$ and $0 < m_1 \le n_1-2$. A complete brute-force factorization of every such sum proved computationally infeasible due to the immense scale of the numbers involved.
We focus on the prime $q=5$, and use properties of primality testing to circumvent this challenge. 

The algorithm first iterates through all valid pairs $(n_1, m_1)$ where the sum $F_{n_1} + F_{m_1}$ is divisible by 5. For each such instance, we calculate the 5-adic valuation $y_1 = \nu_5(F_{n_1} + F_{m_1})$ and define the integer $B := (F_{n_1} + F_{m_1})/5^{y_1}$. To quickly filter for candidates where $B$ is a prime power, we applied a rapid filtering test based on Fermat's Little Theorem. Again, using SageMath's highly optimized \texttt{power\_mod} command, we computed the greatest common divisor $C := \gcd(\text{mod}(2^B-2, B), B)$. This test quickly filters for candidates where $B$ is likely a prime power, as $C>1$ must hold in such cases. For each candidate, we then perform a definitive check by factoring $B$ to verify that it is indeed a prime power.

This optimized computational procedure yielded a total of 263 instances, verifying that in every identified case, $B$ is indeed a prime power with the exponent $x_1$ of $p$ being only $1,2$ or $3$. The resulting primes $p$ are subsequently used in the LLL-reduction, allowing us to establish a more stringent bound on $n_2$ for any remaining potential solutions.

We now use each of these primes $p$ to significantly reduce the upper bound on $n_2$. For a given prime $p$, we consider the linear form in logarithms derived from \eqref{3.1g}. Assuming $d_2\ge 5$, we can write
\begin{align*}
	\left|n_2\log \alpha-x_2\log p-y_2\log 5-\log \sqrt{5}\right|=\left|n_2\log \alpha-x_2\log p-(2y_2+1)\log \sqrt{5}\right|< \frac{6}{\alpha^{d_2}},
\end{align*}
where we used \eqref{eq2.5g}. So, for each $p$ obtained from our algorithm, we consider the approximation lattice
$$\begin{pmatrix}
	1 & 0  & 0 \\
	0 & 1 & 0 \\
	\lfloor M\log (1/p)\rfloor & \lfloor M\log (1/\sqrt5)\rfloor& \lfloor M\log\alpha \rfloor
\end{pmatrix},$$
with $M:= 2\cdot 10^{146}$ and choose $v:=\left(0,0,0\right)$. Now, by Lemma \ref{lem2.5g}, we get 
$$c_1= 10^{-55}\qquad \text{and}\qquad c_2=3.18\cdot 10^{51}.$$
Moreover, by inequalities \eqref{eq:2.4g} and Lemma \ref{lem3.2g}, we have $x_2$, $2y_2+1$, $n_2<2.1n_2<1.1\cdot10^{49}$ so
\[
A_i:=1.1\cdot10^{49},\qquad \text{for}\quad i=1,2,3.
\]
Hence, Lemma \ref{lem2.6g} gives $S=3.36\cdot 10^{98}$ and $T=1.66\cdot 10^{49}$. Since $c_2^2\ge T^2+S$, then choosing $c_3:=6$ and $c_4:=\log\alpha$, we get $d_2\le 457$.

Next, for each $p$ obtained from our algorithm and $d_2\le 457$, we consider the linear form derived
from \eqref{3.3g}, given by
\begin{align*}
	\left|m_2\log \alpha-x_2\log p-y_2\log 5-\log \left(\dfrac{1+\alpha^{-d_2}}{\sqrt5}\right)\right|< \frac{3}{\alpha^{n_2}},
\end{align*}
with the absolute values of integer coefficients not exceeding $n_2 <  5\cdot 10^{48}$.
We consider the approximation lattice
$$ \begin{pmatrix}
	1 & 0  & 0 \\
	0 & 1 & 0 \\
	\lfloor M\log (1/p)\rfloor & \lfloor M\log (1/5)\rfloor& \lfloor M\log\alpha \rfloor
\end{pmatrix},$$
where we set $M :=10^{147}$ and $v :=\left(0,0,-\lfloor M\log \left(\sqrt5 /(1+\alpha^{-d_2})\right) \rfloor\right)$.
By running the algorithm for all primes $p$, we find by Lemma \ref{lem2.5g} that
$$c_1= 10^{-52}\qquad \text{and}\qquad c_2=2.7\cdot 10^{50}.$$
Moreover, by inequalities \eqref{eq:2.4g} and Lemma \ref{lem3.2g}, we have $x_2$, $y_2$, $m_2<n_2<5\cdot10^{48}$ so
\[
A_i:=5\cdot10^{48},\qquad \text{for}\quad i=1,2,3.
\]
Thus, Lemma \ref{lem2.6g} gives $S=7.5\cdot 10^{97}$ and $T=7.51\cdot 10^{48}$. Since $c_2^2\ge T^2+S$, then choosing $c_3:=3$ and $c_4:=\log\alpha$, we get $n_2\le 464$.

To conclude, we use SageMath again to check whether these primes $p$ and $q=5$ have at least two representations of \eqref{eq:main} with each fixed $p^x 5^y$. Still, we do not find any further solutions than those listed in the main result.

\medskip

\noindent{\bf The case $\min\{d_1,d_2\}\ge 34\log n$}

\medskip

In this situation, we have that $\log p\le 13752$ and $n < 10^{53}$ from \eqref{eq:n-bound}. Applying Lemma \ref{lem:lem3.3}, we deduce that
\begin{align*} 
	|d_2-d_1|&\le 1220, \\
	\min\{m_1,m_2\}&\le 1220, \\
	\max\{m_1,m_2\}&\le 2074.
\end{align*}
Let us assume $m \le 2074$. To bound the exponent $y$ of $5$ in the equation $F_n + F_m = p^x 5^y$, we analyze the $5$-adic valuation $\nu_5(F_n + F_m)$. A direct search up to $n < 10^{53}$ is computationally infeasible. Instead, we use a lifting algorithm based on the periodicity of the Fibonacci sequence modulo powers of $5$, which is also similar to a methodology of the computational search in \cite{aded}.

Recall that the Pisano period of the Fibonacci sequence modulo $5^k$ is $\pi(5^k) = 4 \cdot 5^k$. For a fixed $m \in [1, 2074]$, we determine the values of $n$ such that $F_n + F_m \equiv 0 \pmod{5^k}$ recursively, using the algorithm described below and in Appendix \ref{app4}.
\begin{enumerate}[(a)]
	\item \textbf{Base step ($k=1$):} We compute $F_n + F_m \pmod 5$ for $n \in [0, 20)$. We select the indices $n_0$ such that $F_{n_0} + F_m \equiv 0 \pmod 5$ and $n_0 \not\equiv m \pmod 2$. The parity condition ensures we are analyzing the case where $d = n-m$ is odd (as the even case was handled previously).
	\item \textbf{Lifting step:} Given a solution $n_{k-1}$ modulo $4 \cdot 5^{k-1}$, we look for solutions modulo $4 \cdot 5^k$ of the form $n_k = n_{k-1} + j \cdot (4 \cdot 5^{k-1})$ for $j \in \{0, 1, 2, 3, 4\}$. To handle the large indices $n \approx 10^{53}$ efficiently, we use the fact that the Fibonacci sequence can be computed via matrix powers. Specifically, we precompute the matrices $M_k = \mathcal{A}^{4 \cdot 5^k} \pmod{5^{100}}$, where 
	$$\mathcal{A} = \begin{pmatrix} 1 & 1 \\ 1 & 0 \end{pmatrix}.$$
	Instead of calculating $\mathcal{A}^n$ from scratch for each candidate, we compute the state at each lift level $k$ by performing at most four matrix multiplications: $\mathcal{A}^{n_k} = (\mathcal{A}^{4 \cdot 5^{k-1}})^j \cdot \mathcal{A}^{n_{k-1}} \pmod{5^{100}}$. This ``successive 5th powering" approach ensures that the total number of matrix operations is logarithmic ($O(\log n)$), making the search up to $10^{53}$ computationally trivial.
	\item \textbf{Termination:} We iterate this lifting process until the modulus $4 \cdot 5^k$ exceeds the bound $10^{53}$.
\end{enumerate}
Using this algorithm for all $m \in [1, 2074]$, we computed the maximum possible valuation for $n < 10^{53}$. The computation establishes the strict upper bound
\begin{equation*}
	y = \nu_5(F_n + F_m) \le 75.
\end{equation*}

With $y$ bounded by a small constant, we proceed to prove the uniqueness of the solution. Suppose there exist two distinct solutions $(n_1, m_1, x_1, y_1)$ and $(n_2, m_2, x_2, y_2)$ with $n_1 \le n_2$.
From the main equation $F_n + F_m = p^x 5^y$, we derive bounds using the Binet formula. Since $\alpha^{n-2} < F_n < F_n + F_m < F_{n+1} < \alpha^{n}$, we have
$$
\alpha^{n_1-2} < p^{x_1} 5^{y_1} < \alpha^{n_1} \quad \text{and} \quad \alpha^{n_2-2} < p^{x_2} 5^{y_2} < \alpha^{n_2}.
$$
Dividing these inequalities implies
\begin{equation}\label{eq:ratio}
	p^{|x_2-x_1|} \le 5^{|y_2-y_1|} \alpha^{|n_2-n_1|+2}.
\end{equation}
We now know that $|y_2-y_1| \le 75$ and 
\begin{align*}
	|n_2-n_1| &=|(d_2 + m_2)-(d_1+m_1)|\\
	& \le |d_2-d_1| + |m_2-m_1|\\
	&\le 1220+2074  = 3294.
\end{align*}
Since $p > 1000$, inequality \eqref{eq:ratio} constrains the difference $|x_2-x_1|$ to be very small (specifically $|x_2-x_1| \le 247$).

Now, we go back to the matrix \eqref{eq:mat} which has rank 1. So, the matrix of rows
$$
 M=\begin{pmatrix} n_2-n_1 & x_2-x_1 & y_2-y_1 \\ n_2 & x_2 & y_2 + \delta \end{pmatrix},
$$
where $\delta = \log_5(\sqrt{5})=1/2$, must also have rank 1 since we are working with two distinct solutions $(n_1, m_1, x_1, y_1)$ and $(n_2, m_2, x_2, y_2)$. Note that $x_2-x_1$ is not zero since the matrix $M$ has rank 1 (if $x_2=x_1$ and matrix has rank one meaning second row is a multiple of first row, we then get $n_1=n_2$, $y_1=y_2$, contradiction). Thus, the rows are proportional and we can write
$$
\frac{n_2}{n_2-n_1} = \frac{x_2}{x_2-x_1} = \frac{y_2 + \delta}{y_2-y_1}.
$$
This implies 
\begin{align*}
n_2 = (n_2-n_1) \frac{y_2+1/2}{y_2-y_1} < 3294 \cdot \dfrac{76}{1}=250344.
\end{align*}

To complete the proof for this range, we performed an exhaustive computational search for pairs $(n, m)$ such that $F_n + F_m = 5^y p^x$ with $p > 1000$, within the bounds $ n \le 250344$ and $m \le 2074$. 

Given this range for $n$, we optimized the search by first identifying $n \pmod{20}$ such that $F_n + F_m \equiv 0 \pmod 5$. For each such candidate pair, we calculated the $5$-adic valuation $y$ and the remaining part $B = (F_n + F_m)/5^y$. We used SageMath's \texttt{is\_prime\_power()} function to determine if $B$ is a power of a prime $p > 1000$.

The algorithm identified 7143 potential primes $p$. For each such prime, we perform the LLL reduction to reduce the upper bound on $n_2$. For a given prime $p$, we again consider the linear form in logarithms derived from \eqref{3.1g}. If $d_2\ge 5$, we write
\begin{align*}
\left|n_2\log \alpha-x_2\log p-(2y_2+1)\log \sqrt{5}\right|< \frac{6}{\alpha^{d_2}},
\end{align*}
where we used \eqref{eq2.5g}. So, for each $p$ obtained from our algorithm, we consider the approximation lattice
$$\begin{pmatrix}
	1 & 0  & 0 \\
	0 & 1 & 0 \\
	\lfloor M\log (1/p)\rfloor & \lfloor M\log (1/\sqrt5)\rfloor& \lfloor M\log\alpha \rfloor
\end{pmatrix},$$
with $M:= 1.5\cdot 10^{17}$ and choose $v:=\left(0,0,0\right)$. Using Lemma \ref{lem2.5g}, we get 
$$c_1= 10^{-11}\qquad \text{and}\qquad c_2=2.7\cdot 10^{14}.$$
Moreover, by inequalities \eqref{eq:2.4g} and Lemma \ref{lem3.2g}, we have $x_2$, $2y_2+1$, $n_2<2.1n_2<525723$, so
\[
A_i:=525723,\qquad \text{for}\quad i=1,2,3.
\]
Thus, Lemma \ref{lem2.6g} gives $S=8.3\cdot 10^{11}$ and $T=788585$. Since $c_2^2\ge T^2+S$, then choosing $c_3:=6$ and $c_4:=\log\alpha$, we get $d_2\le 16$.

Next, for each $p$ obtained from our algorithm and $d_2\le 16$, we use the linear form derived
from \eqref{3.3g}, as
\begin{align*}
	\left|m_2\log \alpha-x_2\log p-y_2\log 5-\log \left(\dfrac{1+\alpha^{-d_2}}{\sqrt5}\right)\right|< \frac{3}{\alpha^{n_2}},
\end{align*}
with the absolute values of integer coefficients not exceeding $n_2 <  250344$.
We consider the approximation lattice
$$ \begin{pmatrix}
	1 & 0  & 0 \\
	0 & 1 & 0 \\
	\lfloor M\log (1/p)\rfloor & \lfloor M\log (1/5)\rfloor& \lfloor M\log\alpha \rfloor
\end{pmatrix},$$
where we set $M :=10^{18}$ and $v :=\left(0,0,-\lfloor M\log \left(\sqrt5 /(1+\alpha^{-d_2})\right) \rfloor\right)$.
Using Lemma \ref{lem2.5g}, we find that
$$c_1= 10^{-9}\qquad \text{and}\qquad c_2=8.1\cdot 10^{13}.$$
Moreover, by inequalities \eqref{eq:2.4g} and Lemma \ref{lem3.2g}, we have $x_2$, $y_2$, $m_2<n_2<250344$ so
\[
A_i:=250344,\qquad \text{for}\quad i=1,2,3.
\]
Thus, Lemma \ref{lem2.6g} gives $S=1.9\cdot 10^{11}$ and $T=375517$. Since $c_2^2\ge T^2+S$, then choosing $c_3:=3$ and $c_4:=\log\alpha$, we get $n_2\le 21$.

Since $n_1<n_2$, we use SageMath again to check whether these primes $p$ and $q=5$ have at least two representations of \eqref{eq:main} with each fixed $p^x 5^y$. We do not find any further solutions than those listed in the main result.

This completes the Proof of Theorem \ref{th:main}.

\section*{Acknowledgments}

The first author acknowledges the hospitality and support of the Department of Mathematics at the University of Salzburg during his visit in 2024, a period in which part of this project was initiated. This research visit was made possible through the funding provided by the OEAD project of the Africa Uni-Net, project P105 (EREDE). He also extends his sincere gratitude to the Mathematics Division of Stellenbosch University for funding his PhD studies.

\vspace{-0.1cm}
\section*{Addresses}

$ ^{1} $ Mathematics Division, Stellenbosch University, Stellenbosch, South Africa.

Email: \url{hbatte91@gmail.com}

Email: \url{fluca@sun.ac.za}

\noindent 
$ ^{2} $ University of Salzburg, Hellbrunnerstrasse 34/I,
A--5020 Salzburg, Austria

Email: \url{volker.ziegler@sbg.ac.at}

\newpage
\appendix
\section{Appendices}
\subsection{SageMath Code I: Search for solutions}\label{app1}
\begin{verbatim}
N_MAX = 378 
P_MAX = 1000
MIN_SOLUTIONS = 2 
print(f"Starting search with corrected logic for p > q.")
print(f"  -> Search ranges: 2 <= q < p <= {P_MAX} and 0 <= m <= n <= {N_MAX}.\n")

fib = [0, 1]
while len(fib) <= N_MAX:
  fib.append(fib[-1] + fib[-2])
solutions_by_primes = {}

for n in range(2, N_MAX + 1):
  for m in range(n + 1):
    fib_sum = fib[n] + fib[m]
    if fib_sum <= 1:
      continue
    try:
      prime_factors = fib_sum.factor()
    except ValueError:
      continue

    if len(prime_factors) == 2:
      p1_check, x = prime_factors[0]
      p2_check, y = prime_factors[1]
      
      if x > 0 and y > 0:
        if 2 <= p1_check < p2_check <= P_MAX:
          prime_pair = (p2_check, p1_check)
          if prime_pair not in solutions_by_primes:
            solutions_by_primes[prime_pair] = []
          solutions_by_primes[prime_pair].append({
            'n': n,
            'm': m,
            'p_exp': x,
            'q_exp': y,
            'value': fib_sum
          })
print("Finished searching. Now filtering and printing results.\n")
found_any = False
for prime_pair, solutions in sorted(solutions_by_primes.items()):
  if len(solutions) >= MIN_SOLUTIONS:
    found_any = True
    p, q = prime_pair
    print(f"Found {len(solutions)} representations for primes (p, q) = ({p}, {q})")
    for sol in solutions:
      print(f"  F_{sol['n']} + F_{sol['m']} = {sol['value']} 
      = {p}^{sol['p_exp']} * {q}^{sol['q_exp']}")
    print("\n")

if not found_any:
  print(f"No prime pairs (p, q) were found with at least {MIN_SOLUTIONS} 
  solutions in the given ranges.")

\end{verbatim}

\subsection{SageMath Code II: Checking for multiplicative dependence}\label{app2}
\begin{verbatim}
alpha = (1 + sqrt(5)) / 2
sqrt5 = sqrt(5)
R = RealField(200) # Increased precision for better accuracy

MAX_D = 1000
MAX_Q = 1000
EXP_BOUND = 10 # Exponents for the linear form check in [-10, 10]

found_dependencies = []

print("Starting optimized search for multiplicative dependencies.")
print(f"  -> Search ranges: 2 <= d1 < d2 <= {MAX_D} and q <= {MAX_Q}.\n")

fib_cache = {0: 0, 1: 1}
def F(n):
  if n in fib_cache:
    return fib_cache[n]
  if n < 0:
    return 0 # F_n is typically non-negative
  fib_cache[n] = F(n - 1) + F(n - 2)
    return fib_cache[n]

lucas_cache = {0: 2, 1: 1}
def L(n):
  if n in lucas_cache:
    return lucas_cache[n]
  if n < 0:
    return 0 # L_n is typically non-negative
  lucas_cache[n] = L(n - 1) + L(n - 2)
    return lucas_cache[n]
    
def is_multiplicatively_dependent(d1, d2, q, exp_bound=EXP_BOUND):
  if d1 == 10 or d2 == 10: # Condition from the professor's notes
    return False, None
  if q == 5:
    return False, None

  a = R(alpha)
  s5 = R(sqrt5)

  L_values = [
    log(a), # log(alpha)
    log(R(q)), # log(q)
    log((1 + a^(-d1)) / s5),
    log((1 + a^(-d2)) / s5)
  ]

  for b1 in range(-exp_bound, exp_bound + 1):
    for b2 in range(-exp_bound, exp_bound + 1):
      for b3 in range(-exp_bound, exp_bound + 1):
        for b4 in range(-exp_bound, exp_bound + 1):
          if [b1, b2, b3, b4] == [0, 0, 0, 0]:
            continue
          if b3 == 0 or b4 == 0:
            continue

          val = b1*L_values[0] + b2*L_values[1] + b3*L_values[2] + b4*L_values[3]
          if abs(val) < 1e-30:
            return True, (b1, b2, b3, b4)
  return False, None

def get_relevant_factors(d):
  if d <= 1:
    return None
  if d % 2 == 0: # Case for F_{d/2} and L_{d/2}
    n = d // 2
    try:
      factors_F = F(n).factor()
      if any(p[0] > MAX_Q for p in factors_F):
        return None
    except (ValueError, TypeError):
      factors_F = []

    try:
      factors_L = L(n).factor()
      if any(p[0] > MAX_Q for p in factors_L):
        return None
    except (ValueError, TypeError):
      factors_L = []

    return set(p[0] for p in factors_F) | set(p[0] for p in factors_L)
  else: 
    try:
      factors = L(d).factor()
      if any(p[0] > MAX_Q for p in factors):
        return None
      return set(p[0] for p in factors)
    except (ValueError, TypeError):
      return None
      
print("--- Starting search Part (a): d1 <= 12 ---")
for d1 in range(2, 13): # 2 <= d1 <= 12
  for d2 in range(d1 + 1, MAX_D + 1):
    p_factors_d2 = get_relevant_factors(d2)
    if p_factors_d2 is None:
      continue
    p_factors_d1 = get_relevant_factors(d1)
    if p_factors_d1 is None:
      continue
    candidate_q_set = p_factors_d1 | p_factors_d2

    for q in candidate_q_set:
      if q > MAX_Q:
        continue

      dep, relation = is_multiplicatively_dependent(d1, d2, q)
      if dep:
          found_dependencies.append((d1, d2, q, relation))
          print(f"Part (a) -> FOUND DEPENDENCY: (d1, d2, q) = ({d1}, {d2}, {q}), 
          exponents={relation}")
print("\n--- Starting search Part (b): d1 > 13 and d1|d2 ---")
odd_primes = prime_range(3, 1001)

for d1 in range(14, MAX_D + 1):
  for r in odd_primes:
    d2 = d1 * r
    if d2 > MAX_D:
       break

    p_factors_d1 = get_relevant_factors(d1)
    if p_factors_d1 is None:
      continue
      
    p_factors_d2 = get_relevant_factors(d2)
    if p_factors_d2 is None:
      continue
      
    if p_factors_d1.issubset(p_factors_d2):
      new_primes = p_factors_d2 - p_factors_d1

      if len(new_primes) == 1:
        q = list(new_primes)[0]
        if q <= MAX_Q:
          dep, relation = is_multiplicatively_dependent(d1, d2, q)
          if dep:
            found_dependencies.append((d1, d2, q, relation))
            print(f"Part (b) -> FOUND DEPENDENCY: (d1, d2, q) = ({d1}, {d2}, {q}), 
                   exponents={relation}")

print("\n--- Search Complete ---")
if found_dependencies:
  print(f"Found {len(found_dependencies)} total multiplicative dependencies:")
  for d1, d2, q, relation in sorted(found_dependencies):
    print(f"  (d1, d2, q) = ({d1}, {d2}, {q}), exponents={relation}")
else:
  print("No multiplicative dependencies found with the given criteria.")
\end{verbatim}

\subsection{SageMath Code III:  Finding $(p, q)$ prime pairs with $p>1000$ and $q\le 1000$}\label{app3}
\begin{verbatim}
import sys

N1_MAX = 1988
Q_MAX = 1000
EXP_BOUND = 10

# --- Step 1: Redirect standard output to a file ---
output_filename = 'prime_pairs_output.txt'
original_stdout = sys.stdout

try:
  with open(output_filename, 'w') as f:
    sys.stdout = f

    print("Starting search for prime pairs (q, p) with F_n1 + F_m1 = q^y1 * p^x1.")
    print(f"  -> Search ranges: n1 <= {N1_MAX}, m1 <= n1-2, q <= {Q_MAX} (q != 5).\n")

    # --- Step 2: Pre-generate Fibonacci numbers and a list of primes ---
    fib_cache = {0: 0, 1: 1}
    def F(n):
      if n in fib_cache:
        return fib_cache[n]
      if n < 0:
        return 0
      fib_cache[n] = F(n - 1) + F(n - 2)
      return fib_cache[n]

    q_primes = [p for p in primes(2, Q_MAX + 1) if p != 5]

    # --- Step 3: Initialize a data structure to store the solutions ---
    found_prime_pairs = {}
    solution_count = 0

    # --- Step 4: Loop over all triples (n1, m1, q) ---
    for n1 in range(1, N1_MAX + 1):
      if n1 % 100 == 0:
        # Print progress to the file, as there's no console output
        print(f"  ... Processing n1 = {n1} ...")

      # m1 <= n1 - 2
      for m1 in range(0, n1 - 1):
        fib_sum = F(n1) + F(m1)

        if fib_sum <= 1:
          continue

        for q in q_primes:
          if fib_sum % q != 0:
            continue

          # --- Step 5: Calculate the exponent y1 and the number B ---
          y1 = valuation(fib_sum, q)
          B = fib_sum // (q**y1)

          if y1 > 1:
            pass

          # --- Step 6: Quick check for primality of B using powermod ---
          try:
             mod_val = power_mod(2, B, B) - 2
             C = gcd(mod_val, B)
          except ValueError:
             continue

          # --- Step 7: Test if B is a power of a single prime ---
          if C > 1:
            try:
               factors_of_C = factor(C)
               if len(factors_of_C) == 1:
                 p_candidate = factors_of_C[0][0]
                 if p_candidate >1000:
                   is_prime_power = True
                   temp_B = B
                   while temp_B % p_candidate == 0:
                     temp_B //= p_candidate
                   if temp_B == 1:
                     # Found a valid solution.
                     prime_pair = (q, p_candidate)

                     if prime_pair not in found_prime_pairs:
                       found_prime_pairs[prime_pair] = []

                     found_prime_pairs[prime_pair].append({
                       'n1': n1,
                       'm1': m1,
                       'y1': y1,
                       'B': B
                     })
                     solution_count += 1
            except (ValueError, TypeError):
              continue

    # --- Step 8: Print the final results to the file ---
    print("\n--- Search Complete ---")
    print(f"Found {solution_count} total candidate solutions, which fall 
         into {len(found_prime_pairs)} distinct (q, p) pairs.")
    print("\nPrime Pairs (q, p) and their corresponding Fibonacci sums:")

    for prime_pair in sorted(found_prime_pairs.keys()):
       q, p = prime_pair
       solutions = found_prime_pairs[prime_pair]

       print(f"\n  -> Pair (q, p) = ({q}, {p}) found with {len(solutions)} 
            representations:")
       for sol in solutions:
         n1 = sol['n1']
         m1 = sol['m1']
         y1 = sol['y1']
         print(f"     F_{n1} + F_{m1} = {F(n1) + F(m1)} = {q}^{y1} 
               * {p}^{valuation(sol['B'], p)}")

finally:
  # --- Step 9: Restore standard output and print final message to the console ---
  sys.stdout = original_stdout
  print(f"\nSearch complete. The full output has been written to {output_filename}")

\end{verbatim}

\subsection{SageMath Code IV: The case $\min\{d_1,d_2\}\ge 34\log n$}\label{app4}
\begin{verbatim}
def solve_fib_5_adic_lifting():
  print("Starting 5-adic lifting search using successive 5th powering...")

  # Bounds from theory
  MAX_N_BOUND = 10**53  
  MAX_M = 2074
  A = matrix(ZZ, [[1, 1], [1, 0]])
  max_val = 0

  # Precompute lifting matrices: M[j] = A^(4 * 5^j)
  # This ensures O(log n) complexity
  M = [A^20]
  for j in range(80): 
    M.append(M[-1]^5)

  for m in range(1, MAX_M + 1):
    if m % 500 == 0: print(f"Checking m = {m}...")
    Fm = fibonacci(m)

    # Base Step: n modulo 20
    seeds = []
    for n0 in range(20):
      Fn0 = (A^n0)[0,1]
      if (Fn0 + Fm) % 5 == 0 and (n0 % 2) != (m % 2):
        seeds.append((n0, A^n0))

    # Lifting Step
    for n_val, mat_n in seeds:
      curr_period = 20
      v = 1
      j_level = 0

      while curr_period < MAX_N_BOUND:
        # Find current valuation v
        while (mat_n[0,1] + Fm) % 5^(v+1) == 0:
          v += 1
        max_val = max(max_val, v)

        # Lift to next level using precomputed M[j_level]
        found = False
        mod_next = 5^(v+1)
        for k in range(5):
          # mat_next = (A^(curr_period))^k * mat_curr
          mat_cand = (M[j_level]^k) * mat_n
          if (mat_cand[0,1] + Fm) % mod_next == 0:
            n_val += k * curr_period
            mat_n = mat_cand
            curr_period *= 5
            j_level += 1
            found = True
            break
  
         if not found or n_val > MAX_N_BOUND: break

  print(f"\nSearch complete. Maximum 5-adic valuation y = {max_val}")

solve_fib_5_adic_lifting()
	
\end{verbatim}
\end{document}